\documentclass[12pt,reqno]{amsart}

\usepackage[usenames]{color}
\usepackage{graphicx}
\usepackage{amscd}
\usepackage[colorlinks=true,
linkcolor=webgreen,
filecolor=webbrown,
citecolor=webgreen]{hyperref}

\definecolor{webgreen}{rgb}{0,.5,0}
\definecolor{webbrown}{rgb}{.6,0,0}
\usepackage{amsthm}
\usepackage{fullpage}
\usepackage{enumerate}
\usepackage{epstopdf}
\usepackage{graphics,amsmath,amssymb}

\usepackage{amsfonts}
\usepackage{latexsym}
\usepackage{url}

\usepackage{amssymb,amsmath,psfrag}
\usepackage{graphicx}

\newcommand{\seqnum}[1]{\href{http://oeis.org/#1}{\underline{#1}}}

\theoremstyle{plain}
\newtheorem{theorem}{Theorem}[section]

\newtheorem{lemma}[theorem]{Lemma}
\newtheorem{proposition}[theorem]{Proposition}

\theoremstyle{remark}

\newcommand\bigzero{\makebox(0,0){\text{\huge0}}}

\def\adots{
  \mathinner{\mkern1mu\raise1pt\hbox{.}\mkern2mu\raise4pt\hbox{.}
  \mkern2mu\raise7pt\vbox{\kern7pt\hbox{.}}\mkern1mu}}

\begin{document}

\begin{center}	
	\title[Families of Integral Cographs within a  Triangular Arrays]
	{Families of Integral Cographs within a  Triangular Array}
     \author{Hsin-Yun Ching}
	\address{The Citadel }
	\email{hching@citadel.edu}	
	\author{Rigoberto Fl\'orez}
	\address{The Citadel }
	\email{rigo.florez@citadel.edu}
	\author{Antara Mukherjee}
	\address{The Citadel}
	\email{antara.mukherjee@citadel.edu}
\thanks{Several of the results in this paper were found by the first author while working on his undergraduate research project under the guidance of the
second and third authors (who followed the guidelines given  in \cite{FlorezMukherjeeED}).}
\end{center}

\maketitle

\begin{abstract}
The \emph{determinant Hosoya triangle}, is a triangular array where the entries are the determinants of two-by-two Fibonacci matrices.
The determinant Hosoya triangle $\bmod \,2$ gives rise to three infinite families of graphs, that are formed by complete product (join)
of (the union of) two complete graphs with an empty graph. We give a necessary and sufficient condition for a graph from these families   
to be integral.

Some features of these graphs are: they are integral cographs, all graphs have at most five distinct eigenvalues, all graphs are  
either $d$-regular graphs with $d=2,4,6,\dots $ or almost-regular graphs, and some of them are Laplacian integral. Finally  
we extend some of these results to the Hosoya triangle.
\end{abstract}

\section {Introduction}\label{intro}
A graph is integral if the eigenvalues of its adjacency matrix are integers. These graphs are rare and the techniques used to
find them are quite complicated. The notion of integral graphs was first introduced in 1974 by Harary and Schwenk \cite{harary}.
A cograph is a graph that avoids the 4 vertices path as an induced subgraph \cite{Corneil}. In this paper we study an infinite  
family of integral cographs associated with a combinatorial triangle. These families have at most five distinct eigenvalues.

The coefficients of  many recurrence relations are represented using triangular arrangements. These representations give geometric tools 
to study properties  of the recurrence relation. A natural relation between graph theory and recursive relations holds through the   
adjacency matrices from the triangle. A classic example is the relation between the Pascal triangle and Sierpin\'ski graph or Hanoi graph 
that has fractal properties. Recently, some authors have been interested in graphs associated with Riordan arrays. Examples 
of graphs associated with combinatorial triangles can be found in \cite{Baker,Blair,Cheon,CheonJung,deo,koshy2}. 
The aim of this paper is to give another example of a  triangular array that gives rise to families of graphs with good behavior.

We use  $\mathbb{Z}_{>0}$ to denote the set of all positive  integers. The \emph{determinant Hosoya triangle}, in  Table \ref{Symmetric_matrix}, is a triangular array where the entry $H_{r,k}$ with  
$r,k\in \mathbb{Z}_{>0}$  (from left to right) is given by 
\[H_{r,k}=\begin{vmatrix}
F_{k+1} & F_{k}\\
F_{r-k+1} & F_{r-k+2}\\
\end{vmatrix}.\]
For example, the  entry $H_{7,5}$ of $\mathcal{H}$ is given by
\[H_{7,5}=\begin{vmatrix}
F_{6} & F_{5}\\
F_{3} & F_{4}\\
\end{vmatrix}=\begin{vmatrix}
8 & 5\\
2 & 3\\
\end{vmatrix}=24-10=14.\]

\begin{figure}[!ht]\label{symmetric:matrix}
	\includegraphics[width=10cm]{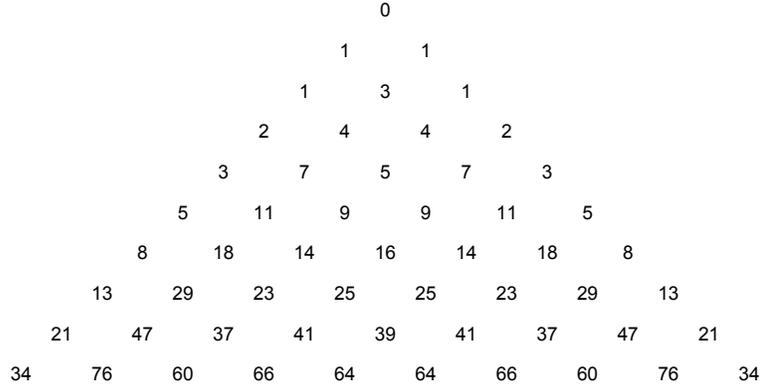}
	\caption{Determinant Hosoya triangle  $\mathcal{H}$.}\label{Symmetric_matrix}
\end{figure}
The determinant Hosoya triangle is symmetric with respect to its median. Therefore, the families of symmetric matrices, that are naturally 
embedded in this triangle, give rise to three infinite families of graphs
(the adjacency matrices are the symmetric matrices $\bmod\; 2$). These graphs are either $d$-regular with $d=2,4,6,\dots $ or almost-regular
graphs. All these three families of graphs have at most five distinct eigenvalues and one of the families is formed by integral graphs. We give
a necessary and sufficient condition to determine whether a family is integral.

The square matrices in the determinant Hosoya triangle are of rank two, so they are the sum of two rank-one matrices. This allows us to classify   
our matrices into three families depending on their size ($n=3t+r$, with $0\le r\le 2$). Their graphs are the complete product of two complete graphs   
with an empty graph. A graph (associated to the determinant Hosoya triangle) is integral if and only if its adjacency matrix is of size $n=3t+1$.

For example, from this triangle we obtain a rank two matrix $\mathcal{S}_{7}$, depicted on the left side of  \eqref{Example:Marices:Mod}. Its rows are 
restrictions of the diagonals of the triangle. This matrix has several interesting properties. For instance, evaluating the entries of this matrix 
$\mod 2$ we obtain an adjacency matrix that gives rise to a regular  subgraph --with five distinct eigenvalues-- that is an integral cograph. The left side in  
Figure \ref{symmetric:graph:Intro} depicts the adjacency graph from $\mathcal{S}_{7} \bmod 2$. Deleting its loops, we obtain the subgraph depicted on  
the right side in  Figure \ref{symmetric:graph:Intro}. 

 \begin{equation}\label{Example:Marices:Mod}
\mathcal{S}_{7}= \left[
\begin{array}{ccccccc}
 0 & 1 & 1 & 2 & 3 & 5 & 8 \\
 1 & 3 & 4 & 7 & 11 & 18 & 29 \\
 1 & 4 & 5 & 9 & 14 & 23 & 37 \\
 2 & 7 & 9 & 16 & 25 & 41 & 66 \\
 3 & 11 & 14 & 25 & 39 & 64 & 103 \\
 5 & 18 & 23 & 41 & 64 & 105 & 169 \\
 8 & 29 & 37 & 66 & 103 & 169 & 272 \\
\end{array}
\right]
\bmod 2 = 
\left[
\begin{array}{ccccccc}
 0 & 1 & 1 & 0 & 1 & 1 & 0 \\
 1 & 1 & 0 & 1 & 1 & 0 & 1 \\
 1 & 0 & 1 & 1 & 0 & 1 & 1 \\
 0 & 1 & 1 & 0 & 1 & 1 & 0 \\
 1 & 1 & 0 & 1 & 1 & 0 & 1 \\
 1 & 0 & 1 & 1 & 0 & 1 & 1 \\
 0 & 1 & 1 & 0 & 1 & 1 & 0 \\
\end{array}
\right]
\end{equation}

\begin{figure}[!ht]
	\includegraphics[scale=0.3]{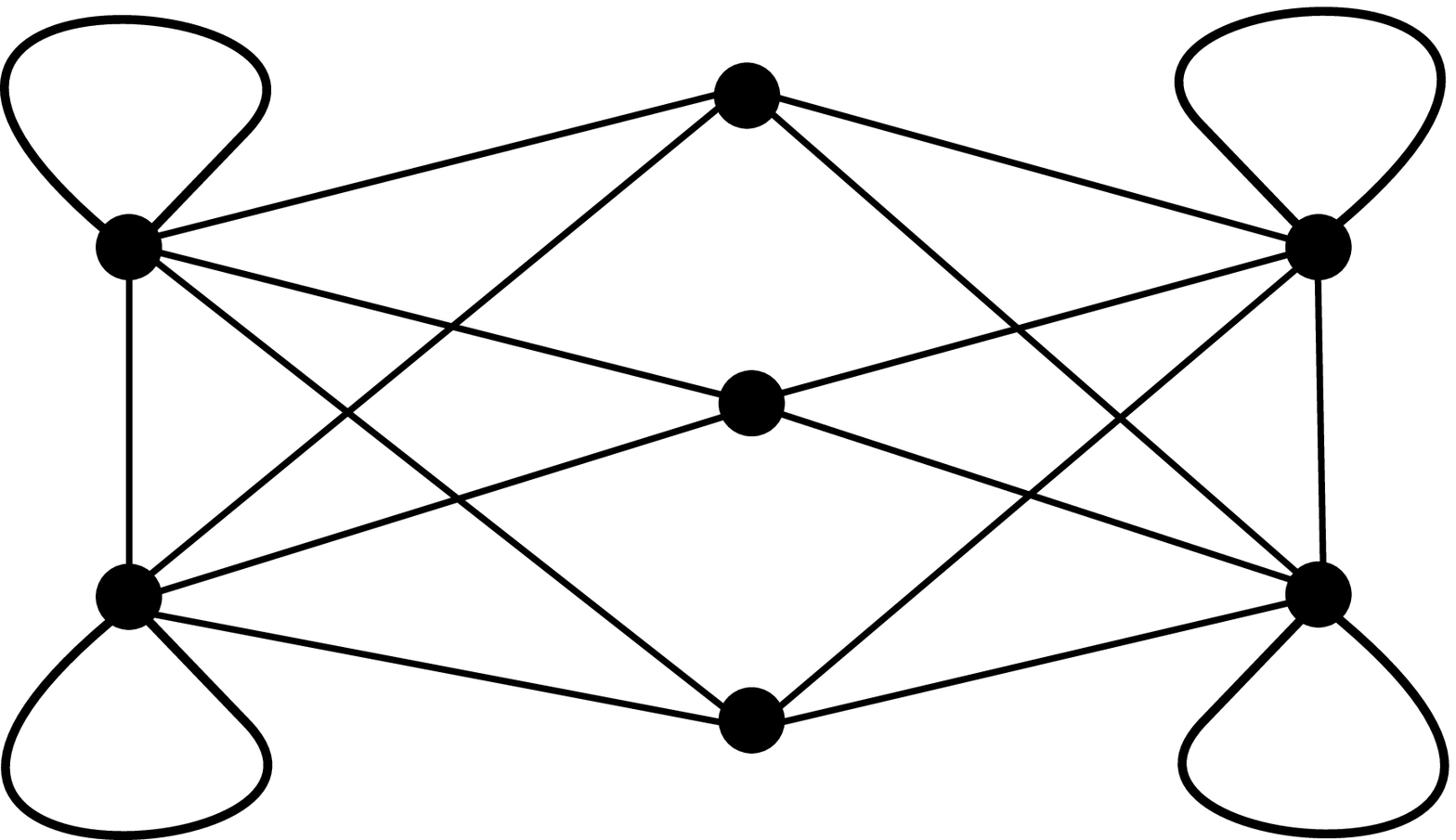} \hspace{2cm} \includegraphics[scale=0.3]{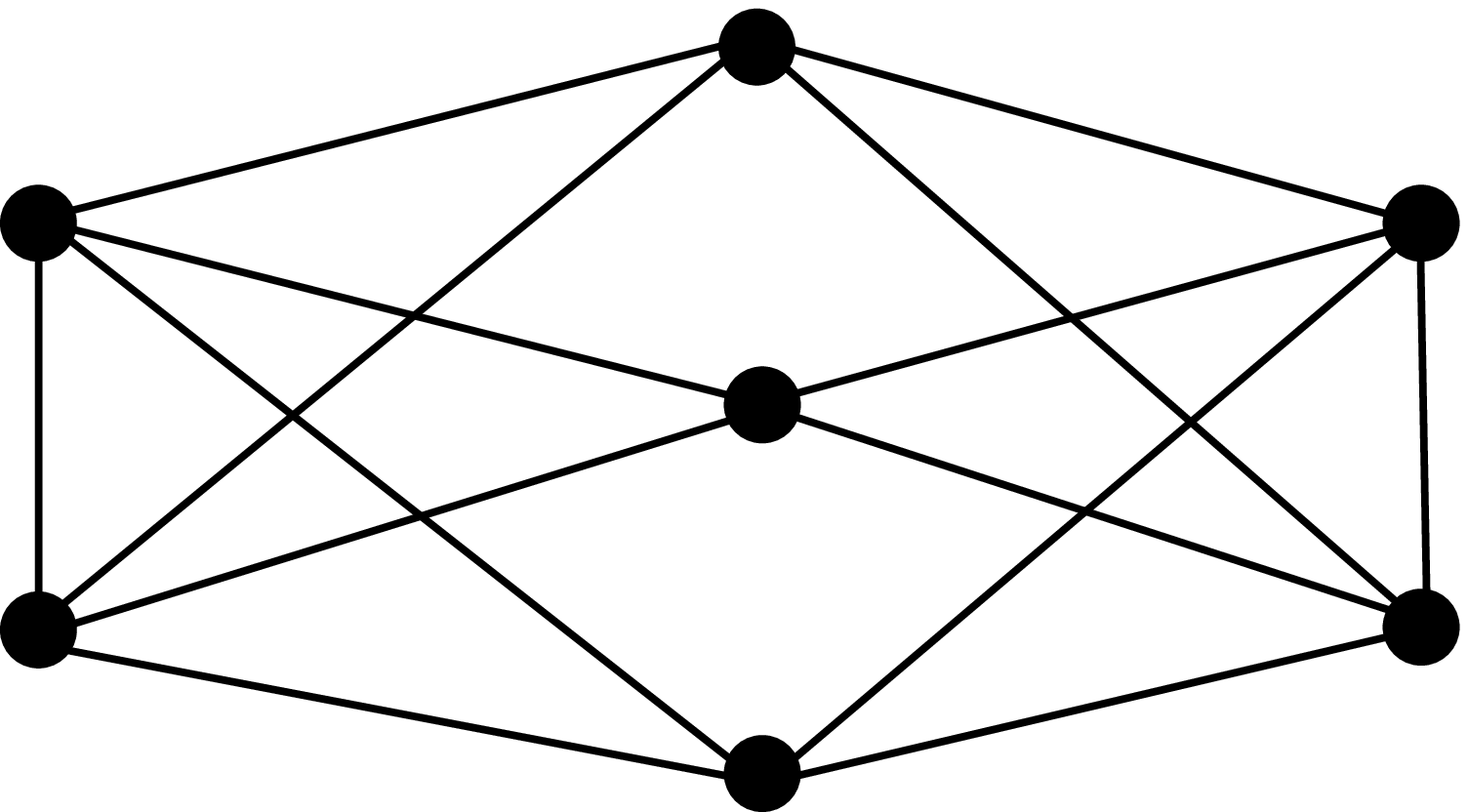}
		\caption{$\mathcal{G}_{7}^{*}=(K_{2}^{*} \sqcup K_{2}^{*}) \nabla \overline{K}_{3}$ \hspace{2cm} $\mathcal{G}_{7}=(K_{2} \sqcup K_{2}) \nabla \overline{K}_{3}$}\label{symmetric:graph:Intro}
\end{figure}

\textbf{Proposition.} \emph{The graph $\mathcal{G}_{w}$ is integral cograph if and only if $w=3t+1$.}

A \emph{join} or \emph{the complete product of two graphs} $G_1$ and $G_2$, denoted by $G_1\nabla G_2$, is defined as the graph obtained  by joining
each vertex of $G_1$ with all vertices of $G_2$.

The main results of this paper show that the graphs that are generated from the matrices in that determinant Hosoya triangle $\bmod\, 2$
are cographs of the form $(K_{n} \sqcup K_{m}) \nabla \overline{K}_{r}$. We use $G_1\sqcup G_2$ to denote the disjoint union of $G_1$ and $G_2$.
We give complete criteria --depending on the embedding of the adjacency matrices within the determinant Hosoya triangle $\bmod\,  2$-- for these
graphs to be integral. Thus, one of the main results of this paper states that a graph $(K_{n}\sqcup  K_{m})\nabla \overline{K}_{r}$ is integral
if and only if $2nr=pq$ and $n-1=p-q$ for some $p,q\in \mathbb{Z}_{>0}$.

Finally, we present a discussion on graphs associated to the \emph{Hosoya triangle}, a triangular array where the entries are products of Fibonacci numbers.

\section{Graphs from matrices in  combinatorial triangles }\label{triangles}

In this section we discuss three families of graphs originated from matrices embedded in the determinant Hosoya triangle. 

\subsection{Determinant Hosoya triangle} \label{triangles:1} 

The \emph{determinant Hosoya triangle} is a triangular array with entries  $\left\{H_{r, k}\right\}_{r, k> 0}$ (taken from left to right) defined recursively by 
\begin{equation}\label{Hosoya:Seq}
 H_{r,k}= H_{r-1,k}+H_{r-2,k}  \; \text{ and } \;
 H_{r,k}= H_{r-1,k-1}+H_{r-2,k-2},
 \end{equation}
with initial conditions $H_{1,1}=0, H_{2,1}= H_{2,2}=1,$ and $ H_{3,2}=3$
where $ r>1 $ and $1\le k \le r$ (see Figure \ref{Symmetric_matrix} on Page \pageref{Symmetric_matrix}). The triangle can also be obtained 
from the generating function $(x+y+xy)/((1-x-x^2)(1-y-y^2))$ (originally discovered by Sloane \cite{sloane}, see \seqnum{A108038}).

Equivalent definitions for the entries of this triangle  are:  
$$H_{r,k}:=F_{k-1}F_{r-k+2}+F_{k}F_{r-k},$$  
which is a variation of  $F_{k+r}=F_{k-1}F_{n}+F_{k}F_{n+1}$ (this identity gives rise to Fibonomial triangle, see Vajda \cite{Vajda}).  
The entry is also a determinant (a formal proof of these two facts is in \cite{BlairRigoAntara})
\[H_{r,k}=\begin{vmatrix}
F_{k+1} & F_{k}\\
F_{r-k+1} & F_{r-k+2}\\
\end{vmatrix}.\]

We now present a result on the divisibility property of the entries of the determinant Hosoya triangle. Note that for every positive $m$,
the entries of $(3m+1)$-th row of $\mathcal{H}$ are always even numbers.

\begin{proposition}\label{gcd:property}
	If $r, k \in \mathbb{Z}_{>0}$, then these hold
	\begin{enumerate}
		\item $F_{\gcd(k+1,r+2)}$ and $ F_{\gcd(k,r+2)}$ divide 
		$$H_{r,k}=\begin{vmatrix}
		F_{k+1} & F_{k}\\
		F_{r-k+1} & F_{r-k+2}\\
		\end{vmatrix},$$
		\item if $r=3t+1$ for $t>1$ and $1\le k\le {\lfloor(r+1)/2\rfloor}$, then
		$$H_{r,k}=\begin{vmatrix}
		F_{k+1} & F_{k}\\
		F_{r-k+1} & F_{r-k+2}\\
		\end{vmatrix} \text{ is even.}$$
	\end{enumerate}
	
\end{proposition}
\begin{proof} Part (1). Since $\gcd(F_{k},F_{k+1})=\gcd(F_{r-k+2},F_{r-k+1})=1$,
	\begin{align*}
	\gcd(F_{k+1}F_{r-k+2}, F_{k}F_{r-k+1}) &= \gcd(F_{k+1},F_{r-k+1}) \gcd(F_{k},F_{r-k+2}) \\
	&= F_{\gcd(k+1,r-k+1)}F_{\gcd(k,r-k+2)}\\
	&= F_{\gcd(k+1,(r+2)-(k+1))}F_{\gcd(k,(r+2)-k)}\\
	&=  F_{\gcd(k+1,r+2)}F_{\gcd(k,r+2)}.
	\end{align*}
	
	Part (2). We prove this part using three cases namely, $k=3m, k=3m+1$, and $k=3m+2$.
If $k=3m$ or $k=3m+2$ for $m>1$, then using Part (1) with $r=3t+1$, we see that $F_{3}=2$ is a factor of $H_{r,k}$. Thus, $H_{r,k}$ is even.
	
	We now prove the case for $k=3m+1$ and $r=3t+1$. Note that $F_{k+1}F_{r-k+2}=F_{3m+2}F_{3t-3m+2}$.
	Rewriting $F_{3m+2}F_{3t-3m+2}$ as $F_{3m+2}F_{3(t-m)+2}$ we see that $F_{k+1}F_{r-k+2}$ is always odd.
	Similarly, we can see that $ F_{k}F_{r-k+1}=F_{3m+1}F_{3(t-m)+1}$, which is always odd.	
	Hence, $H_{r,k}=F_{k+1}F_{r-k+2}-F_{k}F_{r-k+1}$ is even.
\end{proof}

\subsection{Graphs from symmetric matrices} \label{adjacency:graph}

In this section we explore the properties of the graphs from symmetric matrices embedded in $\mathcal{H}$ $\bmod \; 2$.
Note that we are going to analyze the graphs that arise from the triangle in Figure \ref{Symmetric_matrix}.
Replacing the median of this triangle, namely $(0,3,5,16,\ldots)$ by $(0,0,0,0,\ldots)$, the new triangle $\bmod\; 2$, gives rise to a graph without loops.

We start with the definition of a symmetric matrix $\mathcal{S}_{w}$, $w\in \mathbb{Z}_{>0}$ within $\mathcal{H}$,

\begin{equation} \label{eq3}
\mathcal{S}_{w} =  \left[ {\begin{array}{lllll}
	H_{1,1} & H_{2,1} & H_{3,1} & \cdots& H_{w,1} \\
	H_{2,2}  &  H_{3,2}& H_{4,2} & \cdots & H_{w+1,2}\\
	\vdots &\vdots&\vdots&\ddots &\vdots\\
	H_{w,w}  &  H_{w+1,w} & H_{w+2,w} & \cdots & H_{2w-1,w}
	\end{array} } \right]_{w\times w}.
\end{equation}
The matrix $\mathcal{S}_{w}$ can be written as the sum of two rank one matrices. Therefore, the rank of $\mathcal{S}_{w}$ is at most 2.
Thus, $\mathcal{S}_{w}$ has the form
$\mathbf{u_1}^{T}\mathbf{v_1}+\mathbf{u_2}^{T}\mathbf{v_2}$, where $\mathbf{u_1},\mathbf{v_1},\mathbf{u_2},\mathbf{v_2}$
are column vectors with consecutive Fibonacci numbers. These can be seen  located along the sides of $\mathcal{H}$ in     
Figure \ref{Symmetric_matrix}.  Note that the entries of the matrix $\mathcal{S}_{w}\bmod 2$ are $s_{ij}=H_{i,j} \bmod 2$ for $1\le i,j\le w$.   
For example, the matrix $\mathcal{S}_{7}$ given in \eqref{Example:Marices:Mod} on Page \pageref{Example:Marices:Mod} is equal to this 
matrix.

\[
\mathcal{S}_{7}=
\left[
\begin{array}{c}
 0 \\
 1 \\
 1 \\
 2 \\
 3 \\
 5 \\
 8 \\
\end{array}
\right]
\left[
\begin{array}{ccccccc}
 1 & 1 & 2 & 3 & 5 & 8 & 13 \\
\end{array}
\right]
+\left[
\begin{array}{c}
 1 \\
 2 \\
 3 \\
 5 \\
 8 \\
 13 \\
 21 \\ 
\end{array}
\right]
\left[
\begin{array}{ccccccc}
 0 & 1 & 1 & 2 & 3 & 5 & 8 \\
\end{array}
\right].
\]

Let $K_{t}^{*}$ be the complete graph on $t$ vertices with loops at each vertex and recall that $\overline{K}_{t}$ is the empty graph on $t$ vertices.
We show that the graph obtained from the symmetric matrix $\mathcal{S}_{w}\bmod 2$ is given by
\begin{equation}\label{graph:definition}
\mathcal{G}_{w}^{*}=
\begin{cases} (K_{t}^{*}\sqcup K_{t}^{*})\nabla\overline{K}_{t},           & \mbox{ if }\;   w=3t;\\
	(K_{t}^{*}\sqcup K_{t}^{*})\nabla\overline{K}_{t+1},           & \mbox{ if }\;   w=3t+1;\\
	(K_{t}^{*}\sqcup K_{t+1}^{*})\nabla\overline{K}_{t+1},           & \mbox{ if }\;   w=3t+2.
\end{cases}
\end{equation}

For simplicity, we use $\mathcal{G}_{nmr}^{*}$ to denote the graph $(K_{n}\sqcup K_{m})\nabla \overline{K}_{r}$ (again for simplicity, we do not use 
$\mathcal{G}_{n,m,r}^{*}$ that is more natural). Therefore, $\mathcal{G}_{ttt}^{*}$ 
denotes  the graph when $w=3t$, $\mathcal{G}_{tt(t+1)}^{*}$ is the graph when $w=3t+1$, and $\mathcal{G}_{t(t+1)(t+1)}^{*}$ is the graph when $w=3t+2$  
(see Figure \ref{symmetric:graph}).

\begin{proposition}\label{graph:structure}
	Let $\mathcal{G}_{w}^{*}$ be as given in \eqref{graph:definition}. Then the graph of $\mathcal{S}_{w}\bmod 2$ is $\mathcal{G}_{w}^{*}$.
\end{proposition}

\begin{proof}
We prove that for $w=3t+1$, the graph from the matrix is given by $\mathcal{G}_{tt(t+1)}^{*}$, the proof for the cases when $w=3t$ or $3t+2$ are similar so we omit them.
	
Let us consider the matrix $\mathcal{S}_{w}\bmod 2$ for $w=3t+1$. First we establish this notation: we denote by $u_{ij}$ those vertices of the graph
$\mathcal{G}_{tt(t+1)}^{*}$ that correspond to the entries of the $i$-th row and $j$-th column of $\mathcal{S}_{w}\bmod 2$ for $i,j=3k$, $1\le k\le t-1$
(see Figure \ref{symmetric:graph}). Similarly, we denote by $v_{ij}$ the vertices corresponding to entries in the $i$-th row and $j$-th column for
$i,j=3k-1$, $1\le k\le t-1$ and by $z_{ij}$ we denote the vertices corresponding to entries in the $i$-th row and $j$-th column for $i,j=3k+1$, $1\le k\le t$.
Using these notation, it is clear that $\mathcal{G}_{tt(t+1)}^{*}=G_{1}\nabla G_{2}$, such that $G_{1}$ is the union of two complete graphs each of order $t$
on vertices $\{u_{i}\}_{1\le i\le t}$ and $ \{v_{i}\}_{1\le i\le t}$ while $G_{2}=\overline{K}_{t+1}$, the empty graph on the vertices $\{z_{j}\}_{1\le j\le t+1}$.
\end{proof}

\begin{figure}[!ht]
	\includegraphics[scale=0.5]{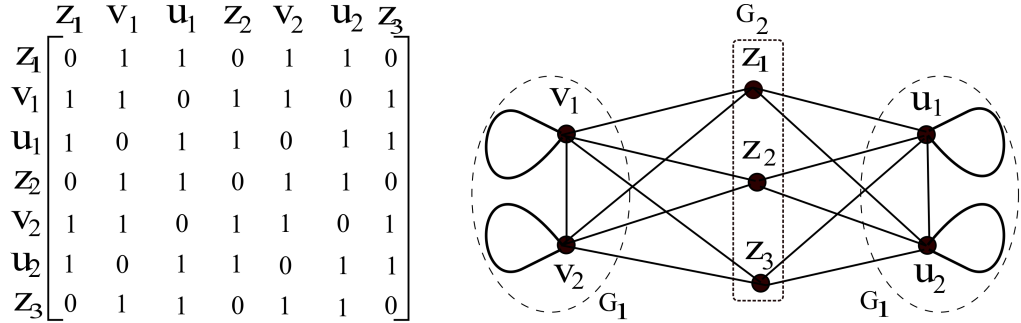}
		\caption{Matrix $\mathcal{S}_{7} \bmod 2$ and its graph $\mathcal{G}_{7}^{*}$.}\label{symmetric:graph}
\end{figure}

Note that, for some results in this paper we also use the graph $\mathcal{G}_{w}=\mathcal{G}_{w}^{*}\setminus \{ \text{all loops} \}$, that is the graph
obtained in Proposition \ref{graph:structure} but without loops. If $K_t$ represents the complete graph (with no loops at any vertex), then the graph
$\mathcal{G}_{w}$ is defined as:

\begin{equation}\label{graph:definition_noloops}
\mathcal{G}_{w}=
\begin{cases} (K_{t}\sqcup K_{t})\nabla\overline{K}_{t},           & \mbox{ if }\;   w=3t;\\
(K_{t}\sqcup K_{t})\nabla\overline{K}_{t+1},         & \mbox{ if }\;   w=3t+1;\\
(K_{t}\sqcup K_{t+1})\nabla\overline{K}_{t+1},           & \mbox{ if }\;   w=3t+2.
\end{cases}
\end{equation}

Once again for simplicity we use the notation $\mathcal{G}_{ttt}$ for the graph without loops when $w=3t$, and similarly, for $w=3t+1$,
we use $\mathcal{G}_{tt(t+1)}$, while for $w=3t+2$, we use the notation $\mathcal{G}_{t(t+1)(t+1)}$ for the graph without loops. See Table
\ref{tabla2} for some examples of this type of graphs.

Note that the empty graph describes the minimum dominating set of the graph $\mathcal{G}_{w}$.

\begin{table}[!ht]
	\begin{tabular}{|l|c|c|c|} \hline
		$t$ &	$3t$ & $3t+1$   &  $3t+2$ \\ \hline  \hline
		2  &$\vcenter{\hbox{\includegraphics[scale=.14]{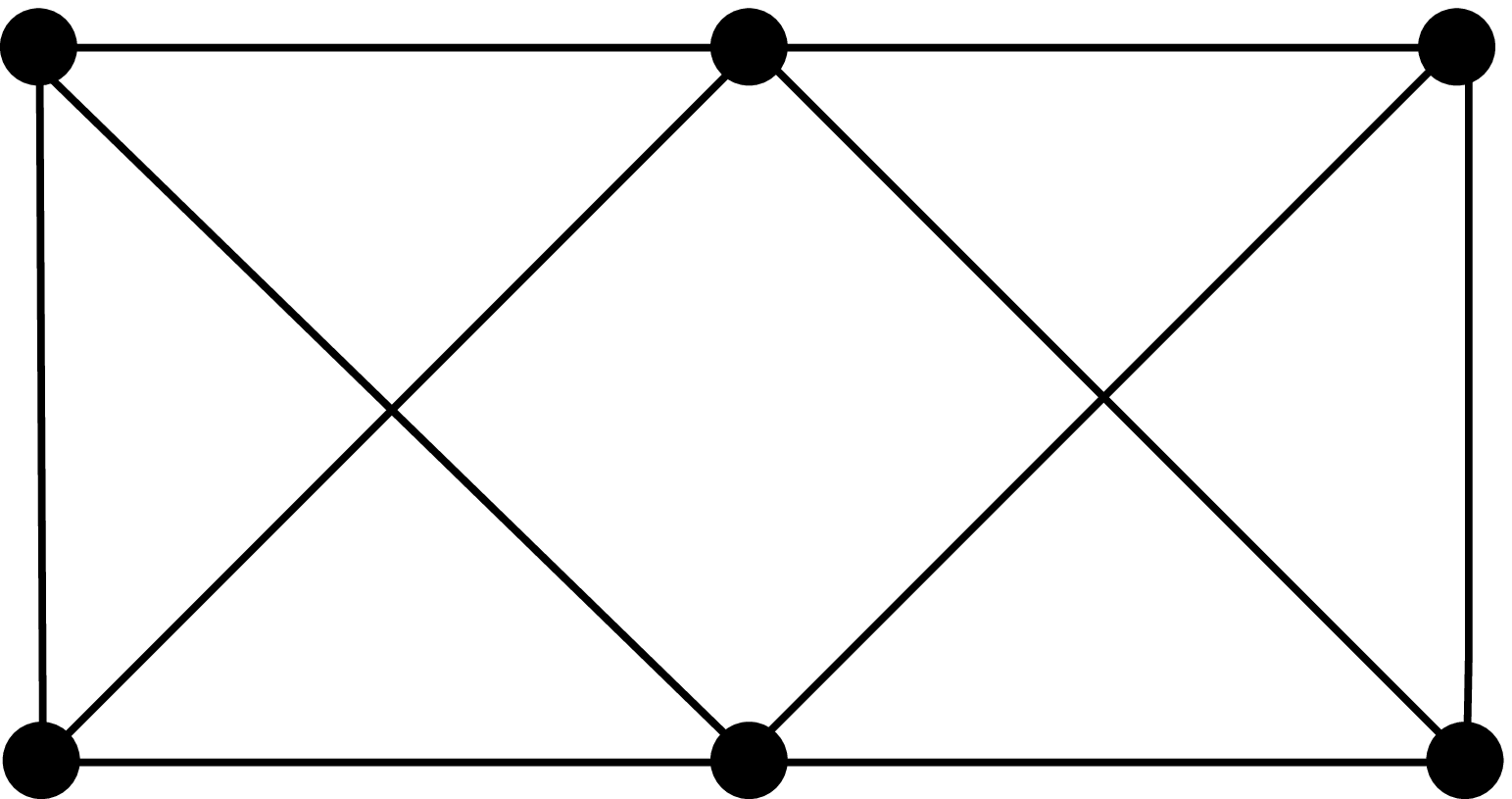}}}$
           &$\vcenter{\hbox{\includegraphics[scale=.14]{Graph1NormalNL.eps}}}$
           &$\vcenter{\hbox{\includegraphics[scale=.14]{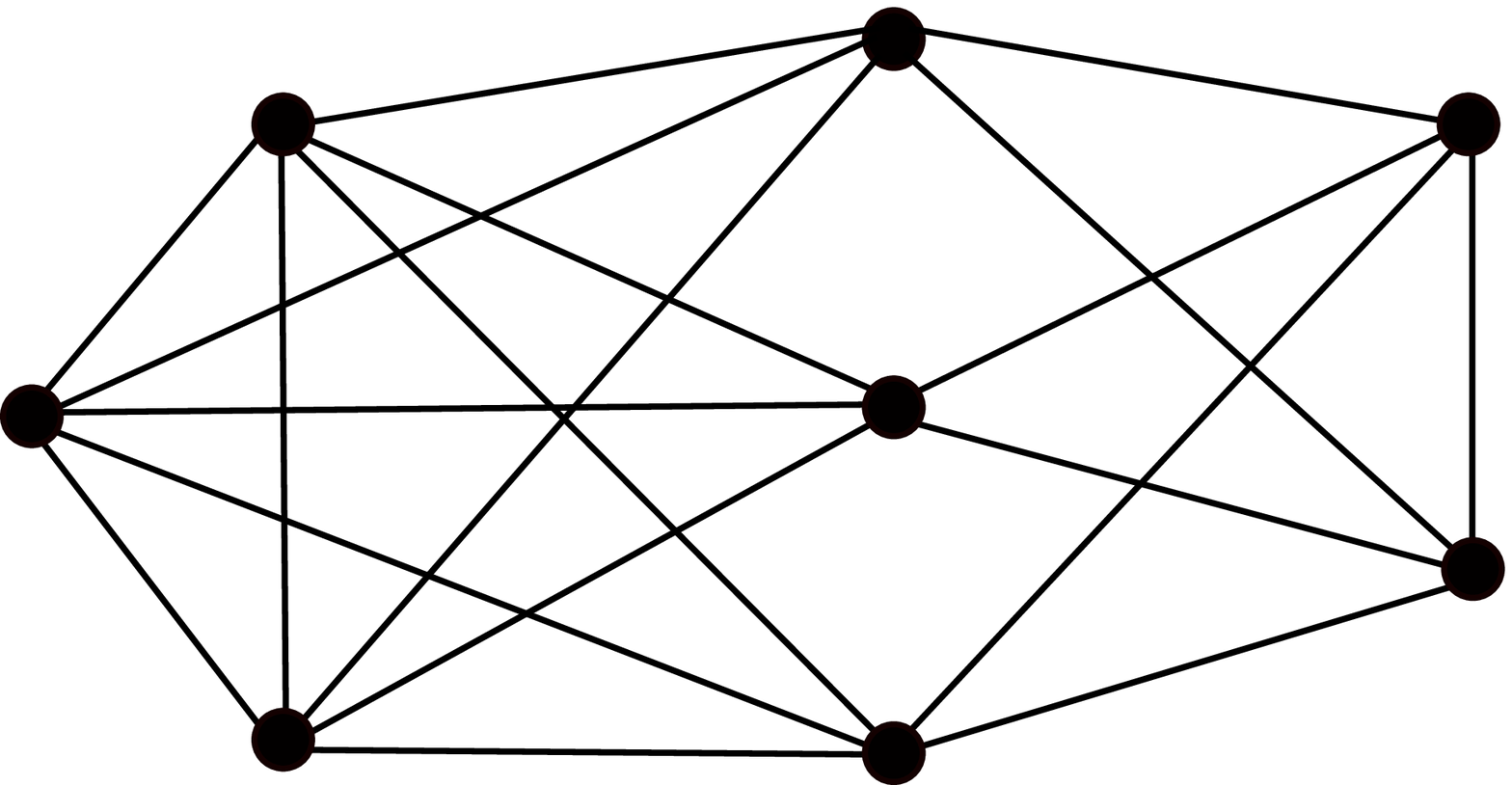}}}$\\ \hline
		3  &$\vcenter{\hbox{\includegraphics[scale=.14]{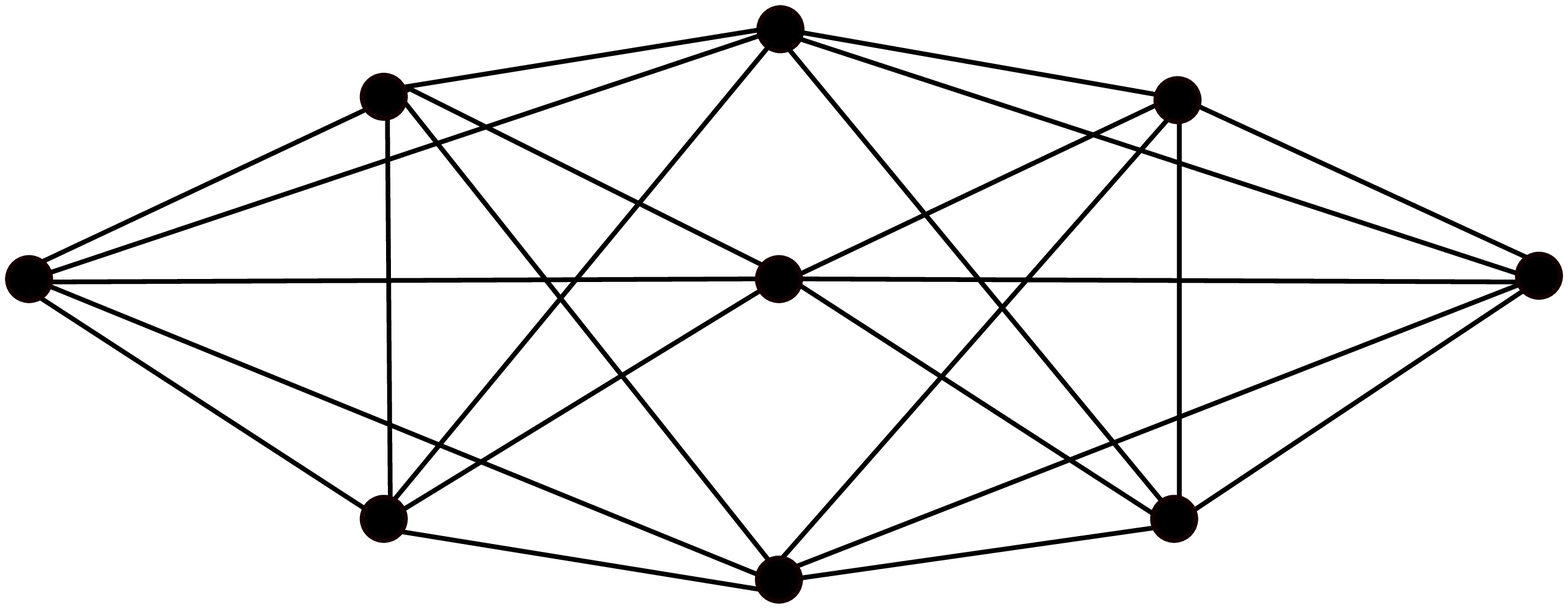}}}$
           &$\vcenter{\hbox{\includegraphics[scale=.14]{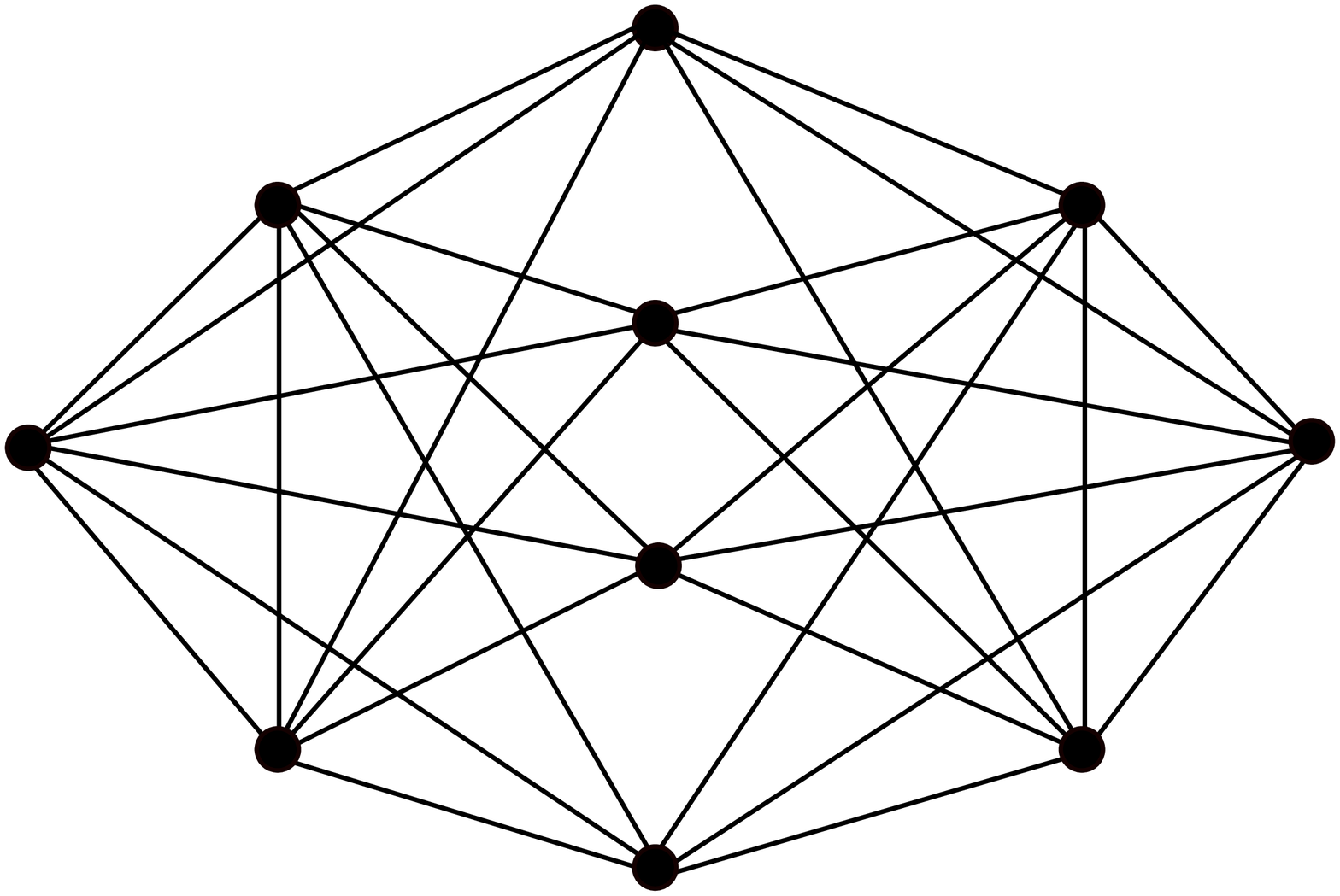}}}$
           &$\vcenter{\hbox{\includegraphics[scale=.14]{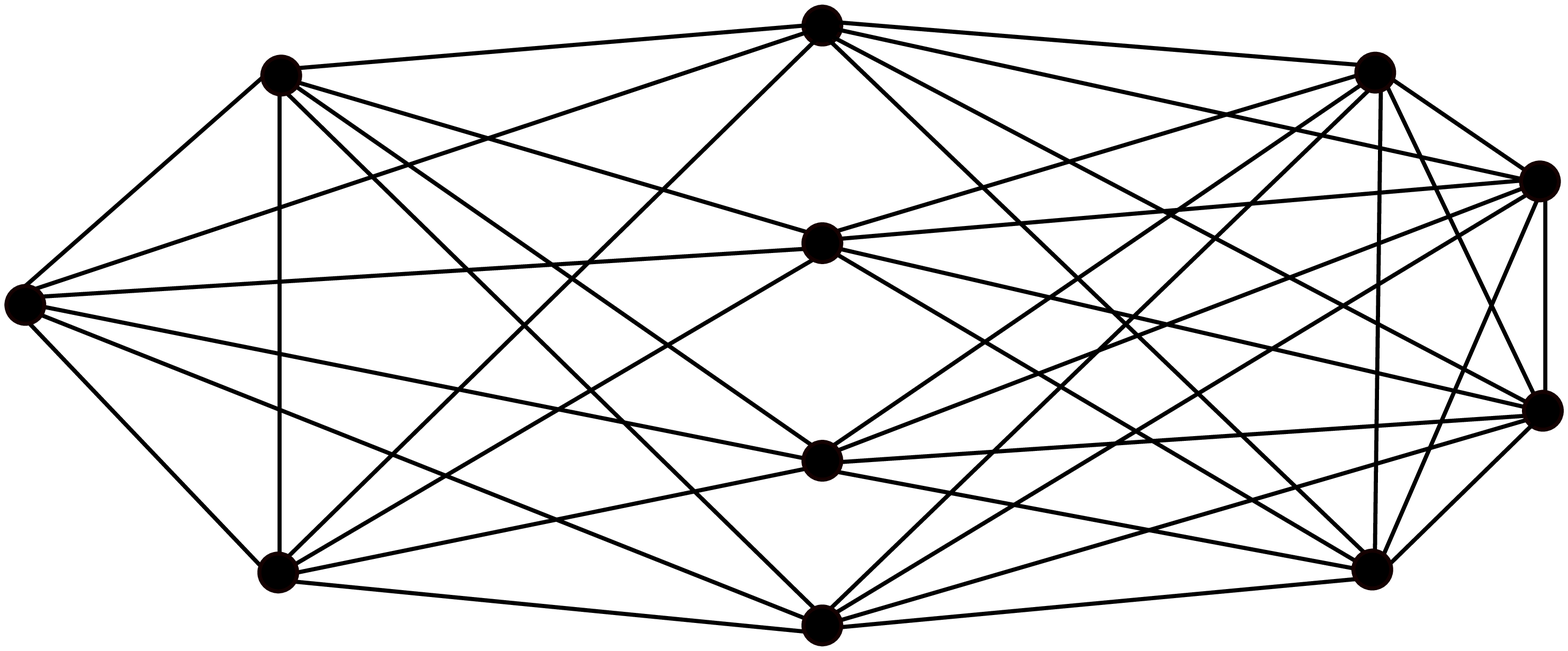}}}$\\ \hline
		4  &$\vcenter{\hbox{\includegraphics[scale=.14]{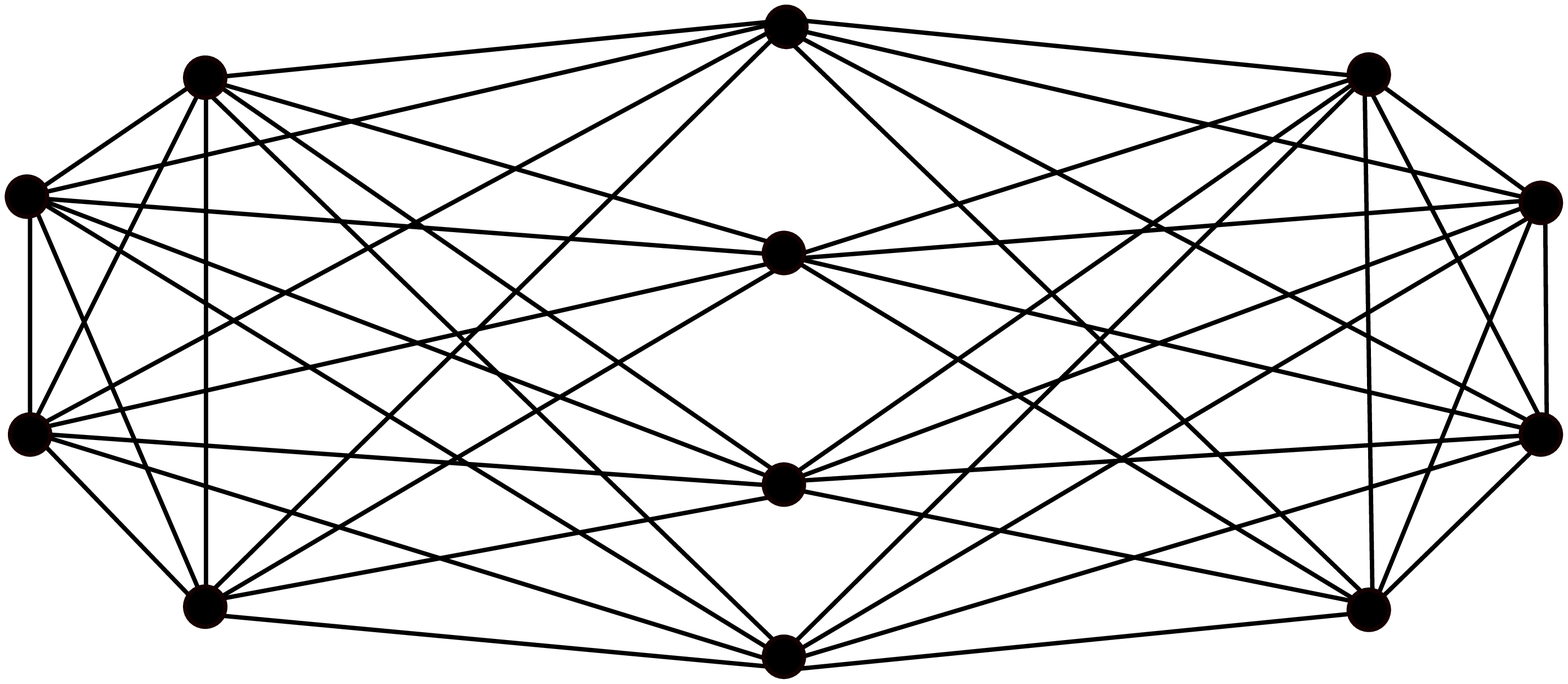}}}$
           &$\vcenter{\hbox{\includegraphics[scale=.14]{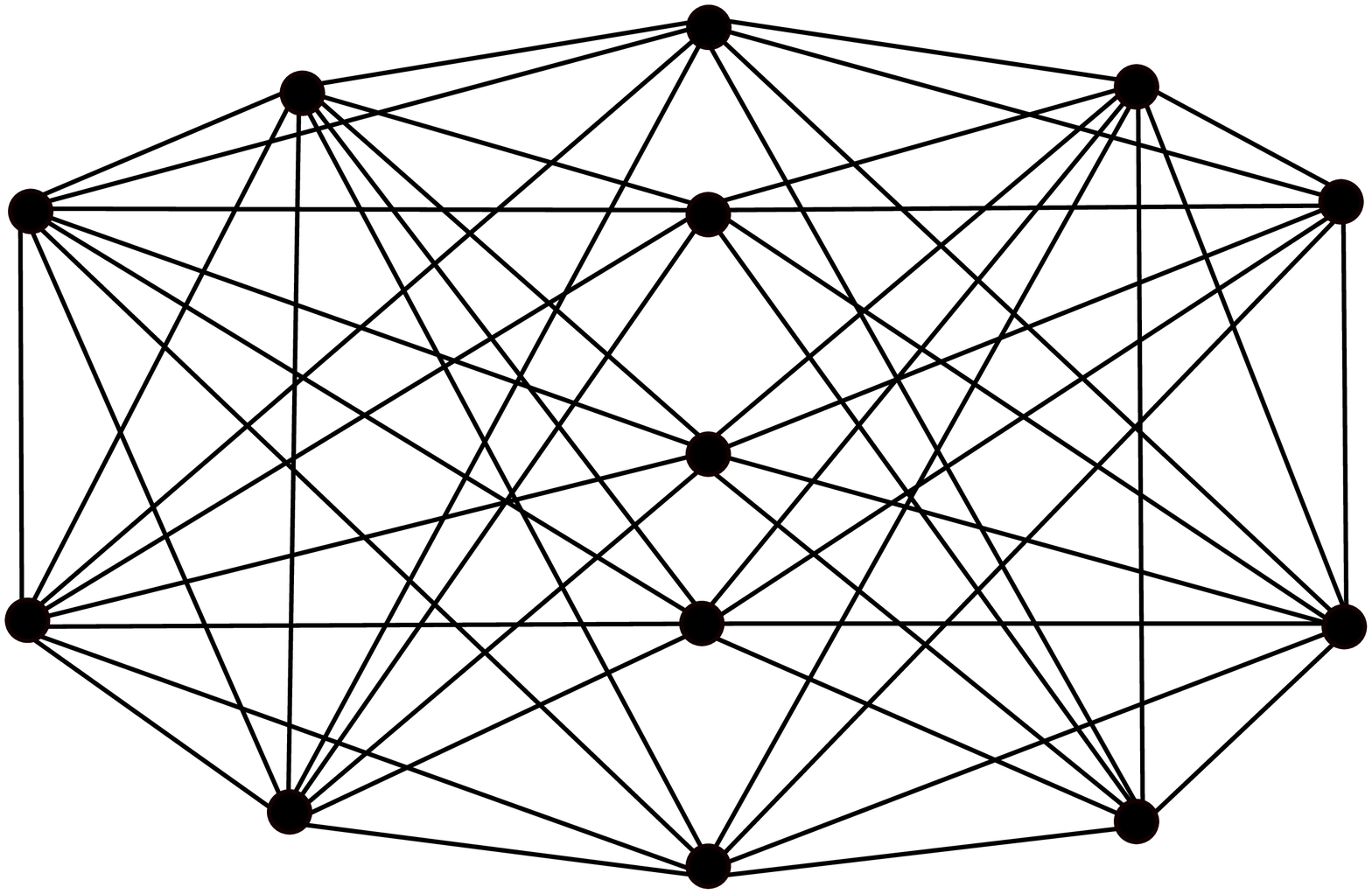}}}$
		   &$\vcenter{\hbox{\includegraphics[scale=.14]{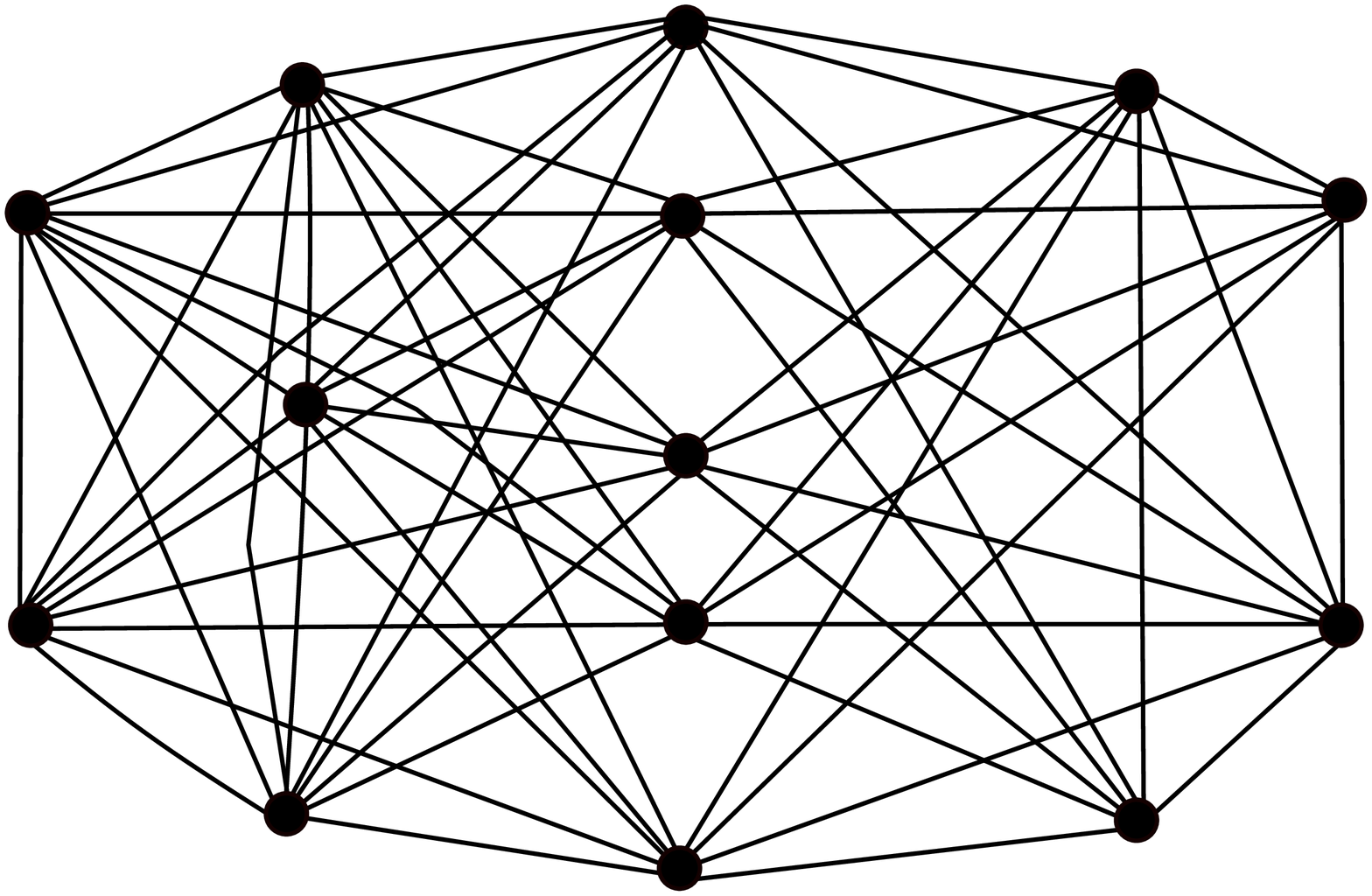}}}$\\ \hline
	\end{tabular}
	\caption{Graphs of $\mathcal{G}_{w}$.} \label{tabla2}
\end{table}

It is clear from the definition of join of graphs that $\mathcal{G}_{nmr}=(K_{n}\sqcup K_{m})\nabla \overline{K}_{r}$ is connected for $n,m,r>0$.
It is easy to verify that a connected graph is $d$-regular if and only if the largest eigenvalue is $d$ and $[1,1,\ldots,1]$ is an eigenvector of the graph.
We use this result to show that the graph  $\mathcal{G}_{tt(t+1)}$ is regular.

In many papers, $\Delta$ and $\delta$ are used to denote the maximum and minimum degree of a graph, we use the same notation in the following results.

A \emph{cograph} is defined recursively as follows: any single vertex graph is a cograph, if $G$ is a cograph, then so is its complement graph 
$\overline{G}$, and if $G_1$ and $G_2$ are cographs, then so is their disjoint union $G_1\sqcup G_2$. Equivalently, a cograph is a graph which  
does not contain the path on 4 vertices as an induced subgraph \cite{Corneil}.  The proof to the following proposition is straightforward 
and we omit it.  

\begin{proposition}\label{Propositon:cograph}
	The graphs  $\mathcal{G}_{w}^{*}$ and $\mathcal{G}_{w}$, as defined in \eqref{graph:definition} and \eqref{graph:definition_noloops}, 
	 are cographs.
\end{proposition}

Now we give some consequences of this proposition. Royle \cite{Royle} proves that the rank of a cograph $X$ is equal to the number of 
distinct non-zero rows of its adjacency matrix. Using this result  we have as corollaries of Proposition \ref{Propositon:cograph} these results. 
The rank of the adjacency matrix of $\mathcal{G}_{t(t+1)(t+1)}$  is   $2 (t + 1)$; the rank of the adjacency matrix of $\mathcal{G}_{ttt}$  is 
$2 t + 1$; and the rank of the adjacency matrix of $\mathcal{G}_{tt(t+1)}$  is  $2 t + 1$. 

Two vertices are \emph{duplicate} if their open neighborhoods are the same \cite{Biyikoglu}. Two vertices in 
$\overline{K}_{t}$ or  $\overline{K}_{t+1}$ with $t>1$, are duplicate vertices of a graph  $\mathcal{G}_{w}$ as in \eqref{graph:definition_noloops}.   
 Therefore, by \cite{Biyikoglu,Royle} we know that $0$ is an eigenvalue of  $\mathcal{G}_{w}$   (similarly we have that $-1$ is an is an 
 eigenvalue of $\mathcal{G}_{w}$) see Proposition \ref{Corollary:1} for the whole set of eigenvalues. 
 
In the following part we present families of regular and almost-regular connected graphs and their characteristics.
 
\begin{theorem}\label{theorem:regular}
	The graph $\mathcal{G}_{w}$ is regular if and only if $w=3t+1$. Moreover, $\mathcal{G}_{tt(t+1)}$ is $2t$-regular.
\end{theorem}

\begin{proof}
First we show that the graph $\mathcal{G}_{tt(t+1)}$ is regular.
The adjacency matrix of $\mathcal{G}_{tt(t+1)}$ is $A=[a_{ij}]$, where $a_{ii}=0$ and $a_{ij}=\begin{vmatrix}
	F_{i+1} & F_{i}\\
	F_{j} & F_{j+1} \\
	\end{vmatrix} \bmod 2$ for $i\ne j$. Therefore,
	
	\[a_{ij}=
	\begin{cases}	
 (F_{i+1}F_{j+1}-F_{i}F_{j}) \bmod 2, & \text{ for } i\ne j;\\
 0, & \text{ for } i= j.
	\end{cases}
	\]
	
{\bf Claim}. The entry $a_{ij}$ of $A$, where $a_{ij}\in \xi=\{a_{1,3k+1}, a_{2,3k},a_{3,3k-1},\ldots, a_{3k+1,1}\}$,  is equal to 
$0 \bmod 2$ for $1\le k\le t$ and 1 otherwise. (Sometimes we use $a_{i,j}$ instead of $a_{ij}$ to avoid ambiguities.)

Proof of Claim. Consider the entry $a_{i,3k+2-i}=F_{i+1}F_{3k+2-i+1}-F_{i}F_{3k+2-i}$. We prove the claim by three cases,  
$i=3s-1$, $i=3s,$ and $i=3s+1$ for $s>1$.  Indeed, if $i=3s-1$ it is easy to see that $F_{\gcd(i+1,3k+2-i)}$ and  
$F_{\gcd(i,3k+2-i+1)}$ divide $a_{i,3k+2-i}$, and $F_{\gcd(i+1,3k+2-i)}$ is even (every third Fibonacci number is even).   
If $i=3s$, then in this case also $F_{\gcd(i+1,3k+2-i)}$ and $F_{\gcd(i,3k+2-i+1)}$ divide $a_{i,3k+2-i}$ and  $F_{\gcd(i,3k+2-i+1)}$ is even.  
Finally, if $i=3s+1$, by Proposition \ref{gcd:property}, $a_{i,3k+2-i}$ is even. However, the other entries $a_{ij}$ of $A$ for $a_{ij}\notin \xi$ 
are all odd. The matrix $A$ is therefore of the form

\begin{equation} \label{eq6}
A =  \left[ {\begin{array}{lll|l}
	A(K_{t}) & &\;\bigzero &  \\
	&&& \\
	&&&\;\;\;\; J^{T}\\
  \;\;\;\; \bigzero&& A(K_{t}) &  \\
&&&\\
\hline
&&&\\
&	J &   &A(\overline{K}_{t+1})\\
\end{array} } \right]_{(3t+1)\times(3t+1)}
\end{equation}
where $A(K_{t})$ and $A(K_{t+1})$ are the adjacency matrices of the complete graphs $K_{t}$ and $K_{t+1}$ and  $J=[j_{rl}]_{(t+1)\times t}$ with $j_{rl}=1$
for $1\le r\le (t+1)$ and $1\le l\le t$ ($J$ is the matrix of $1$'s).

From $A$ we see that the graph is connected. It is easy to verify that $\lambda=2t$ is the largest eigenvalue of the $A$ and that
$[1,1,\ldots,1]_{3t+1}$ is an eigenvector of $A_{w}$. This proves that $\mathcal{G}_{tt(t+1)}$ is a $2t$-regular graph.

Conversely, we show that if $w\ne 3t+1$, then $\mathcal{G}_{w}$  is not regular.
We prove this using cases.

{\bf Case $w=3t+2$}. By the definition of join of graphs we know that $\mathcal{G}_{t(t+1)(t+1)}$ is connected with maximum degree $\Delta=2t+1$ and
minimum degree $\delta= 2t$. The number of vertices of degree $\Delta$ is $2t$ while the number of vertices of degree $\delta$ is $t$.
Since $\Delta-\delta=1$, $\mathcal{G}_{t(t+1)(t+1)}$ is almost-regular.

{\bf Case  $w=3t$}. Recall that $\mathcal{G}_{ttt}$ is connected and $\Delta=2t$ and $\delta= 2t-1$. The number of vertices of degree $\Delta$ is $t$
while the number of vertices of degree $\delta$ is $2t$. Once again, since $\Delta-\delta=1$, $\mathcal{G}_{ttt}$ is almost-regular.
This completes the proof.
\end{proof}	
	
We observe that the graph $\mathcal{G}_{ttt}^{*}$ (defined in \eqref{graph:definition}) has loops at all vertices that have degree $\delta=2t-1$.
Similarly, the graph $\mathcal{G}_{t(t+1)(t+1)}^{*}$ has loops at all vertices that have degree $\delta=2t$. If we have $\mathcal{G}_{tt(t+1)}^{*}$,
then the entries of the diagonal $a_{ii}$ are even if $i=3k+1$ for $0\le k\le t$. This implies that the number of loops in $\mathcal{G}_{tt(t+1)}^{*}$ is $2t$.

Note. The graph $\mathcal{G}_{tt(t+1)}$ (defined in \eqref{graph:definition_noloops}) is not vertex transitive for $t>1$ and vertex transitive when $t=1$.
The graphs $\mathcal{G}_{tt(t+1)}$ with $t>1$ and  $\mathcal{G}_{ttt}$ with $t\ge 1$, have two orbits each. In fact, in both cases the vertices
of the graph that are part of $(K_{t}\sqcup K_{t})$, all belong to one orbit while the vertices of the empty graphs, are in a second orbit.
In the case of the graph $\mathcal{G}_{t(t+1)(t+1)}$ with $t\ge 1$, the vertices are split up into three orbits. The vertices of $(K_{t}\sqcup K_{t+1})$
that have degree $2t$ are in the first orbit while the vertices of $(K_{t}\sqcup K_{t+1})$ that have degree $(2t+1)$ are all in the second orbit and finally,
the vertices of the empty graph $\overline{K}_{t+1}$ are in the third orbit. This shows that these graphs are non-asymmetric.

\section{Integral graphs}\label{Background on IntegralGraphs}
A graph is called \emph{integral} if all its eigenvalues are integers. Now we summarize some results of
Bali\'{n}ska et al. \cite{balinska}. Let $G_i$ be a graph with $n_i$ vertices and $r_i$ regularity for $i=1,2$.
\begin{equation}\label{perfect_square}
G_{1}\nabla G_{2} \text{ is integral $\iff$ $G_{1}$ and $G_{2}$ are integral and} (r_{1}-r_{2})^{2}+4n_{1}n_{2} \text{ is a perfect square.}
\end{equation}

In 1980 Cvetkovi\'{c} \textit{et al.} \cite{cvetkovi} (see also \cite{hwang}) proved that the characteristic polynomial $P_{G_{1}\nabla G_{2}}(\lambda)$ of $G_{1}\nabla G_{2}$ equals
\begin{equation}\label{eqn:1}
(-1)^{n_{2}}P_{G_1}(\lambda)P_{\overline{G}_{2}}(-\lambda-1)+(-1)^{n_{1}}P_{G_2}(\lambda)P_{\overline{G}_{1}}(-\lambda-1)-(-1)^{n_1+n_2}P_{\overline{G}_{1}}(-\lambda-1)P_{\overline{G}_{2}}(-\lambda-1).
\end{equation}
 In particular, if $G_{i}$ is $r_{i}$-regular in the result above, then the characteristic polynomial of $G_{1}\nabla G_{2}$ is given by

 \begin{equation*}
P_{G_{1}\nabla G_{2}}(\lambda) =\frac{P_{G_{1}}(\lambda)P_{G_{2}}(\lambda)}{(\lambda-r_{1})(\lambda-r_{2})}((\lambda-r_{1})(\lambda-r_{2})-n_{1}n_{2}).
\end{equation*}

Let $A(G_1)_{m\times m}$ and $A(G_2)_{n\times n}$ be the adjacency matrices of $G_1$ and $G_2$ and let $a,c\in \mathbb{R}^{m}$ and $b,d\in \mathbb{R}^{n}$, 
where $ad^{T}$ and $bc^T$ are matrices with $1$'s as their entries.
From Zhang \textit{et al.} \cite{Zhang} (see also \cite{hwang}), we know that the adjacency matrix of $G_1\nabla G_2$ is

\begin{equation} \label{eqn:3}
M =  \left[ {\begin{array}{ll}
	A(G_{1}) & ad^{T} \\
	bc^T & A(G_{2})
	\end{array} } \right].
\end{equation}
In addition, if $\widetilde{A_{1}}=ac^{T}-A(G_{1})$ and  $\widetilde{A_{2}}=bd^{T}-A(G_{2})$, then the characteristic polynomial of $M$ is given by

\begin{equation}\label{eqn:4}
P_{M}(\lambda)=(-1)^{m}P_{\widetilde{A_{1}}}(-\lambda)P_{A(G_{2})}(\lambda)+(-1)^{n}P_{A(G_{1})}(\lambda)P_{\widetilde{A_{2}}}(-\lambda)-(-1)^{m+n}P_{\widetilde{A_{1}}}(-\lambda)P_{\widetilde{A_{2}}}(-\lambda).
\end{equation}

\subsection{An infinite family of integral graphs}\label{main_result}
In this section we give a more general result that will be important for proving results about the graphs embedded in the combinatorial triangles in this paper.
In particular, we present the necessary and sufficient condition for the graph $\mathcal{G}_{nmr}:=(K_{n}\sqcup K_{m})\nabla \overline{K}_{r}$ to be integral.

The first three parts of the following lemma are well known. Since the characteristic polynomial is multiplicative under disjoint union of graphs,
the last part of the lemma is a product of the characteristic polynomials of  $K_{n}$ and  $K_{m}$. So, we omit the complete proof of the lemma.

\begin{lemma}\label{BasicCharacteristicPoly} If $n$ is a positive integer, then

\begin{enumerate}
\item \label{CompleteN} the characteristic polynomial of $K_{n}$ is given by $p(\lambda)=(-1)^n(\lambda-(n-1))(\lambda+1)^{n-1}$.

\item \label{BipartiteN}The characteristic polynomial of $K_{m,n}$ is given by $p(\lambda)=(-1)^{m+n}\lambda^{m+n-2}(\lambda^2-nm)$.

\item \label{EmptyN} The characteristic polynomial of $\overline{K}_{n}$ is given by $p(\lambda)=(-1)^n\lambda^n$.

\item  \label{DisjointUnionN}The characteristic polynomial of  $K_{m}\sqcup K_{n}$ is given by
$$p(\lambda)=(-1)^{m+n}(\lambda-(n-1))(\lambda-(m-1))(\lambda+1)^{m+n-2}.$$

\end{enumerate}

\end{lemma}

\begin{proposition} \label{Main:thm1} Let $n,m,r\in \mathbb{Z}_{>0}$ and let  $P(\lambda)$ be the characteristic polynomial of $\mathcal{G}_{nmr}$.
\begin{enumerate}
\item  \label{Main:thm1:part1}  If $n\ne m$, then $P(\lambda)$ is equal to 
 $$(-1)^{m+n+r}\lambda^{r-1}(\lambda+1)^{m+n-2}[\lambda^{3}-(m+n-2)\lambda^{2}-((m+n)(r+1)-mn-1)\lambda-((m+n)r-2mnr)].$$
 \item  \label{Main:thm1:part2} If $n=m>1$, then $$P(\lambda)=(-1)^r\lambda^{r-1} (\lambda+1)^{2(n-1)}(\lambda-(n-1))(\lambda^2-(n-1)\lambda-2nr).$$
 Moreover,  $\mathcal{G}_{nnr}$ is integral if and only if $2nr=pq$ and $n-1=p-q$ for some $p, q\in \mathbb{Z}_{>0}$.
 \end{enumerate}
\end{proposition}

\begin{proof}  The proof of Part  \eqref{Main:thm1:part2} is a direct application of Part \eqref{Main:thm1:part1}, therefore, it is omitted. So, we prove Part \eqref{Main:thm1:part1}. We start by computing the characteristic polynomial  
$P(\lambda)$ of $\mathcal{G}_{nmr}$. In fact, to obtain $P(\lambda)$, we substitute the appropriate characteristic polynomials from  
Lemma  \ref{BasicCharacteristicPoly} into \eqref{eqn:1}.

Since $\mathcal{G}_{nmr}:=(K_{n}\sqcup K_{m})\nabla \overline{K}_{r}$, we set $G_1:= K_{n}\sqcup K_{m}$, $G_2:=\overline{K}_{r}$ in \eqref{eqn:1}.
Therefore, their complement graphs are $\overline{G_1}=K_{m,n}$ and $\overline{G_2}=K_r$. This and Lemma \ref{BasicCharacteristicPoly} imply that
$$P_{\;\overline{G_1}}(-\lambda-1)=(\lambda+1)^{m+n-2}((\lambda+1)^2-mn)\quad \text{ and } \quad P_{\;\overline{G_2}}(-\lambda-1)=(\lambda+r)(\lambda)^{r-1}.$$
Substituting this and $P_{G_1}(\lambda)$, $P_{G_2}(\lambda)$ (after applying Lemma  \ref{BasicCharacteristicPoly}) into \eqref{eqn:1}
gives that $P(\lambda)$ is equal to

\begin{multline*}
(-1)^{m+n+r}(\lambda+1)^{m+n-2}\lambda^{r-1}[(\lambda-n+1)(\lambda-m+1)(\lambda+r)+ \lambda((\lambda+1)^2-mn)\\
-(\lambda+r)((\lambda+1)^2-mn)].
\end{multline*}

Simplifying the polynomial further we obtain that $P(\lambda)$ is equal to
	\[
	(-1)^{m+n+r}\lambda^{r-1}(\lambda+1)^{m+n-2}[\lambda^{3}-(m+n-2)\lambda^{2}-((m+n)(r+1)-mn-1)\lambda-r(m+n-2mn)].
	\]	
This completes the proof.
\end{proof}

\begin{proposition} \label{Main:thm1:Moreover} Let $n,r\in \mathbb{Z}_{>0}$. 
 $\mathcal{G}_{nnr}$ is integral if and only if $2nr=pq$ and $n-1=p-q$ for some $p, q\in \mathbb{Z}_{>0}$.
\end{proposition}

\begin{proof} 
Let $p,q\in \mathbb{Z}_{>0}$ such that $pq=2nr$ and $p-q=n-1$. This implies that characteristic polynomial $P(\lambda)$ in Proposition \ref{Main:thm1} Part \eqref{Main:thm1:part1} has linear factors.
	
Conversely, if $\mathcal{G}_{nnr}=(K_{n}\sqcup K_{n})\nabla \overline{K}_{r}$ with either $2nr\ne pq$ or $n-1\ne p-q$ for every $p,q\in \mathbb{Z}_{>0}$,
then $(\lambda^2-(n-1)\lambda-2nr)$ does not have linear factors.
\end{proof}

Another way to see why $\mathcal{G}_{nnr}$ is integral, is as follows.  Let us first note that $(K_{n}\sqcup K_{n})$ is $(n-1)$-regular with $2n$ vertices and
$\overline{K}_{r}$ is 0-regular with $r$ vertices. In addition, from  \cite{balinska,barik} we have that $(K_{n}\sqcup K_{n})$ and $\overline{K}_{r}$
are both integral. Therefore, substituting $(n-1)$, $0$ and $2n$ into \eqref{perfect_square} or in the expression $(r_{1}-r_{2})^{2}+4n_{1}n_{2}$ shows that
it is a perfect square. Therefore, $(r_{1}-r_{2})^{2}+4n_{1}n_{2} =(n-1)^2+8nr$.  Since $pq=2nr$ and $p-q=n-1$ (for $p,q\in \mathbb{Z}_{>0}$),
we have $(n-1)^2+8nr=(p+q)^2$. This proves that $\mathcal{G}_{nnr}$ is integral.

\subsection{Integral Graphs from Combinatorial Triangles}
In this section we present several results that provide the criteria for the graphs given in \eqref{graph:definition} and \eqref{graph:definition_noloops},
to be integral. In particular, we present families of integral and non-integral graphs with at most five distinct eigenvalues.

\subsection{Graphs without loops}
In this section we give characteristic polynomials of graphs with only linear factors over the set of integers. We also give characteristic
polynomials of certain graphs that do not factor completely over the set of integers.
Finally, we provide a necessary and sufficient for the graphs to be integral.
 Note that these graphs are defined in \eqref{graph:definition_noloops} on Page \pageref{graph:definition_noloops}.

\begin{proposition}\label{Corollary:1}
For $t>0$, these hold
\begin{enumerate}[(a)]
  \item \label{poly:1} the characteristic polynomial of $\mathcal{G}_{tt(t+1)}$ is given by

  \[P(\lambda)=
\begin{cases} \lambda^2(\lambda+2)(\lambda-2),   & \mbox{ if }\;   t=1;\\
\lambda^{t}(\lambda-2t)\left((\lambda+1)^2-t^{2}\right)(\lambda+1)^{2(t-1)}, & \text{ if } t>1.
\end{cases}
\] 
 \item \label{poly:2} The characteristic polynomial of $\mathcal{G}_{ttt}$ is given by

\[P(\lambda)=
	\begin{cases} \lambda(\lambda^{2}-2),                                           & \mbox{ if }\;   t=1;\\
	\lambda^{t-1}(\lambda+1)^{2(t-1)}(\lambda-t+1)(\lambda^{2}-(t-1)\lambda-2t^{2}),& \text{ if } t>1.
	\end{cases}
\] 
 \item \label{poly:3} The characteristic polynomial of $\mathcal{G}_{t(t+1)(t+1)}$ is given by

\[P(\lambda)=
	\begin{cases} \lambda(\lambda+1)(\lambda^{3}-\lambda^{2}-6\lambda+2),                            & \mbox{ if }\;   t=1;\\
	\lambda^{t}(\lambda+1)^{2t-1}(\lambda^{3}+(1-2t)\lambda^{2}-(t^{2}+4t+1)\lambda+(2t^{2}-1)(t+1)), & \text{ if } t>1.
	\end{cases}
\] 

\end{enumerate}

\end{proposition}

The proof of Proposition \ref{Corollary:1} is straightforward from  Proposition \ref{Main:thm1}. The \emph{energy} of a graph $G$ is the sum 
of the absolute values of the eigenvalues of $G$. A direct consequence of Proposition \ref{Corollary:1} Part \eqref{poly:2} is that the energy of  
$\mathcal{G}_{tt(t+1)}$ is $6t-2$, for $t>0$. 

We now give the proof  for the necessary and sufficient condition for the graphs from the symmetric matrices in $\mathcal{H}\bmod 2$ to be integral.
We start with a lemma that will be important for showing that $\mathcal{G}_{ttt}$ and $\mathcal{G}_{t(t+1)(t+1)}$ are not integral graphs.

\begin{lemma}\label{not:integral:new}
If $t\in \mathbb{Z}_{>0}$, then these polynomials do not factor with monic linear factors in $\mathbb{Z}[x]$
\begin{enumerate}[(a)]
       \item \label{part:a} $p_{1}(x)=x^{2}+(1-t)x-2t^{2}$,
       \item \label{part:b} $p_{2}(x)=x^{3}+(1-2t)x^{2}-(t^{2}+4t+1)x+(2t^{3}+2t^{2}-t-1)$,
       \item \label{irr:polyn1} $p_{3}(x)=x^{2}-tx-2t(t+1)$,
       \item \label{irr:polyn2} $p_{4}(x)=x^{3}-(2t+1)x^{2}-(t+1)^{2}x+2t(t+1)^{2}$.
     \end{enumerate}
\end{lemma}

\begin{proof}
Proof of Part \eqref{part:a}. Since the discriminant of $p_{1}(x)$ and $(1-t)$ have distinct parity, using the quadratic formula we have that any
root of the polynomial $p_{1}(x)$ has two as the denominator.

 Proof of Part \eqref{part:b}. We prove this part by contradiction. Let us assume that $q\in \mathbb{Z}$ is a root of $p_{2}(x)$. So, $q$ must divide
 the independent coefficient $(2t^{3}+2t^{2}-t-1)=(2t^2-1)(t+1)$. Clearly, $q\ne 0$ and $1\le q\le (2t^2-1)(t+1)$. Since $p(q)=0$,
 we have $q^{3}+(1-2t)q^{2}-(t^{2}+4t+1)q+(2t^{3}+2t^{2}-t-1)=0$.
Rewriting the cubic expression we have,	
\begin{equation}\label{eqn:roots}
       q^{3}+(1-2t)q^{2}-(t^{2}+4t+1)q = (1-2t^2)(t+1).
\end{equation}
	
Now we consider the following five cases:

{\bf Case $q=1$.} Substituting $q=1$ in \eqref{eqn:roots} we obtain that $1-6t-t^2>(1-2t^2)(t+1) $ if $t>1$. It is now easy to see that $1-6t-t^2$
is greater than the right hand side of \eqref{eqn:roots} for $t>1$.
		
{\bf Case $1<q<t+1$.} Since both $q$ and $t$ are integers, there is an integer $h\ge 2$ such that $t+1=q+h$.
We now prove that, substituting  $q=t+1-h$ in \eqref{eqn:roots} the left hand side of \eqref{eqn:roots} is greater than its right hand side.
Thus, $q^{3}+(1-2t)q^{2}-(t^{2}+4t+1)q=(1-2t^3-2t-5t^2)-h^3+4h^2-4h+h^2t+2ht^2$.
We observe that for $h\ge 2$,
		$(1-2t^3-2t-5t^2)-h^3+4h^2-4h+h^2t+2ht^2\ge (1-2t^3-2t-5t^2)+4t+4t^2$.
		
It is easy to check that,	
		  \[(1-2t^3-2t-5t^2)+4t+4t^2 =1-2t^3-t^2+2t > 1-2t^3-2t^2+t \text{ for } t>1.\]
		
{\bf Case $q=t+1$.} Substituting $q=t+1$ in the left hand side of \eqref{eqn:roots} we obtain $1-2t^3-2t-5t^2$. It is once again easy to see
that  $1-2t^3-5t^2-2t<1-2t^3-2t^2+t$ for $t>1$.
	
{\bf Case $(t+1)<q\le (2t-1)$.} Here we claim that the left side of \eqref{eqn:roots} is less than the right side. To prove this we observe that since
$q\le 2t-1$,
\[q^{3}+(1-2t)q^{2}-(t^{2}+4t+1)q\le 1+2t-7t^2-2t^3.\]
	
Next we observe that $1+2t-7t^2-2t^3<1-2t^3-2t^2+t$ for $t>1$.
	
{\bf Case $q>(2t-1)$.} Let $q=2t-1+h$ where $h>1$ is an integer.  We prove that the left hand side of \eqref{eqn:roots} is greater than its right hand side.
Indeed, substituting $q=2t-1+h$ in the left hand side of \eqref{eqn:roots} and simplifying we obtain
	$$q^{3}+(1-2t)q^{2}-(t^{2}+4t+1)q=(1+2t-7t^2-2t^3)+h^3-2h^2+4h^2t-8ht+3ht^2.$$	
Next we observe that for $h\ge 2$,
		$$(1+2t-7t^2-2t^3)+h^3-2h^2+4h^2t-8ht+3ht^2\ge (1+2t-7t^2-2t^3)+6t^2.$$	
Finally, it is easy to check that
	\[(1+2t-7t^2-2t^3)+6t^2= 1-2t^3-t^2+2t >1-2t^3-2t^2+t \text { for } t>1.\]
This completes the proof.	

Proof of Part  \eqref{irr:polyn1}. It is straightforward using the quadratic formula. Observe that the discriminant of $p_{3}(x)$ is between two consecutive perfect squares,
for $t>0$. Thus, $(3 t+1)^2<9 t^2+8 t<(3 t+2)^2$.

Proof of Part \eqref{irr:polyn2}. Suppose, by contradiction, that $q\in \mathbb{Z}$ is a root of $p_{4}(x)$. Therefore, $q$ divides the independent coefficient of $p_{4}(x)$.
Thus  $p\mid 2t(t+1)^{2}$. (Clearly, $q\ne 0$ and $1\le |q|\le 2t(t+1)^2$.)
This implies that, $q^{3}-(2t+1)q^{2}-(t+1)^{2}q+2t(t+1)^{2}=0$. Adding $(t+1)^2$ to either side of this equation and factoring the left hand side we have	
\begin{equation}\label{eq10}
	 (q-(2t+1))(q^{2}-(t+1)^2)=(t+1)^2.
\end{equation}
	
We now consider the following four cases.
	
{\bf Case $q=\pm (t+1)$ or $q=\pm t$.} Substituting $q=\pm (t+1)$  in \eqref{eq10} gives that the left side equals zero. That is a contradiction, because the
right side is non-zero. It is also easy to see that substituting $q=\pm t$ in \eqref{eq10} we obtain a similar contradiction.
	
{\bf Case $(t+1)< |q|\le (2t+1)$.} Since $[q-(2t+1)]\le 0$ and $[q^{2}-(t+1)^2]> 0$, the left side of \eqref{eq10} is less than or equal to zero.
That is a contradiction.
	
{\bf Case $ |q|> (2t+1)$.}  There is an integer $h\ge 1$ such that $|q|=2t+1+h$. Substituting this value of $q$ in $(q^{2}-(t+1)^2)$, we have that
	$[2(t+1)(t+h)+(t+h)^{2}]>(t+1)^2$. This proves that when we substitute the value of $q$ in the left hand side of \eqref{eq10} we have that
	$H[2(t+1)(t+h)+(t+h)^{2}]\ne(t+1)^2$, where $H=(\pm(2t+1+h)-(2t+1))$.
	
{\bf Case $ 0<|q|<t$.} We analyze the case in which $1\le q<t$, the case $-t<q\le-1$ is similar and it is omitted. There is an integer $h\ge 1$
such that  $q=t-h$. Substituting this value of $q$ in the left hand side of \eqref{eq10} and simplifying we obtain
	  $(t+h+1)(t+1+h(2t-h)+t)$.  Since $h<t$, we have $(2t-h)+t\ge1$. This implies that $(t+1+h(2t-h)+t)>(t+1)$. Therefore,
	  $(t+h+1)(t+1+h(2t-h)+t)>(t+1)^{2}$. This completes the proof.
\end{proof}

\begin{proposition} \label{Main:thm3}
Let $\mathcal{G}_{w}$ be as in \eqref{graph:definition_noloops} on Page \pageref{graph:definition_noloops}.
The graph $\mathcal{G}_{w}$ is integral if and only if $w=3t+1$.
\end{proposition}

\begin{proof}
We first prove the necessity of $w$ for $\mathcal{G}_{w}$ by contradiction. Suppose that $\mathcal{G}_{w}$ is integral and that  $w \ne 3t+1$.
Therefore, $w=3t$ or $w=3t+2$. However, by Proposition \ref{Corollary:1} parts (b) and (c), and Lemma \ref{not:integral:new}, we know that the
characteristic polynomials in both cases do not have all integer roots. Therefore, $\mathcal{G}_{ttt}$ and $\mathcal{G}_{t(t+1)(t+1)}$ are not integral.
That is a contradiction. Thus, $w=3t+1$.

We now use cases to prove the sufficiency of $w$ for $\mathcal{G}_{w}$.

{\bf Case $w=3t+2$.} Using Proposition \ref{Corollary:1} Part (c) and Lemma \ref{not:integral:new} Part (b) it is clear that $\mathcal{G}_{t(t+1)(t+1)}$
is not an integral graph.

{\bf Case $w=3t$.} We analyze the first line in \eqref{graph:definition_noloops}.
That is, the graph $\mathcal{G}_{ttt}=(K_{t}\sqcup K_{t})\nabla \overline{K}_{t}$. We know that the complete graph $K_{t}$ is integral and regular,
therefore $(K_{t}\sqcup K_{t})$ is integral and regular. In addition, the empty graph $ \overline{K}_{t}$ is integral and regular.

Substituting $r_{1}:=(t-1)$ and $r_{2}:=0$ (the regularities) and $n_{1}:=2t$ and $n_{2}:=t$  (the number of vertices of $(K_{t}\sqcup K_{t})$ and
$\overline{K}_{t}$) into \eqref{perfect_square}, we have that, $(r_{1}-r_{2})^{2}+4n_{1}n_{2}=9t^{2}-2t+1$. Note that $(3t-1)^2<9t^{2}-2t+1<(3t)^2$ for $t>0$.
Thus, $9t^{2}-2t+1$ is between two consecutive squares. So it is not a perfect square for $t>0$. This proves that $\mathcal{G}_{ttt}$ is not an integral graph.

{\bf Case $w=3t+1$.}  We analyze the second line in \eqref{graph:definition_noloops}. That is, the graph
$\mathcal{G}_{tt(t+1)}=(K_{t}\sqcup K_{t})\nabla \overline{K}_{t+1}$. Since  $(K_{t}\sqcup K_{t})$  and $\overline{K}_{t+1}$ are integral and regular,
we have that substituting
$r_{1}:=(t-1)$ and $r_{2}:=0$ (the regularities) and $n_{1}:=2t$ and $n_{2}:=t+1$ (the number of vertices of $(K_{t}\sqcup K_{t})$ and $\overline{K}_{t+1}$)
into \eqref{perfect_square} we obtain  $(r_{1}-r_{2})^{2}+4n_{1}n_{2}=(3t+1)^2$, a perfect square. This implies that $\mathcal{G}_{tt(t+1)}$ is an integral graph.
\end{proof}
Note.  Since $G_1:=K_{t}\sqcup K_{t}$ and $G_2:=\overline{K}_{l}$ are both regular using \cite[Table 2]{barik} we have that $\mathcal{G}_{tt(t+1)}$ 
is integral and that $\mathcal{G}_{ttt}$ is not integral. These give alternative proofs  for Case $w=3t+1$ and Case $w=3t$ in the previous proof. However,
for $\mathcal{G}_{t(t+1)(t+1)}$ we cannot use the result from \cite[Table 2]{barik}.

\subsection{Graphs with loops}
Let us recall that in Proposition \ref{graph:structure} on Page \pageref{graph:structure} we showed that the graphs from
$\mathcal{S}_{w} \bmod 2$ for $w=3t+r$, $0\le r\le 2$ are given by $\mathcal{G}_{ttt}^{*}$, $\mathcal{G}_{tt(t+1)}^{*}$ and $\mathcal{G}_{t(t+1)(t+1)}^{*}$
(defined in \eqref{graph:definition} on Page \pageref{graph:definition}).		
In this section we provide the characteristic polynomials of these graphs. We also provide a necessary and sufficient condition for the graphs
to be integral, this will be given as the proof of Theorem  \ref{Main:thm4}.

We start the section with two lemmas that will be important for showing that the characteristic polynomials of
$\mathcal{G}_{tt(t+1)}^{*}$ and $\mathcal{G}_{t(t+1)(t+1)}^{*}$ do not have all integer roots.

\begin{proposition}\label{Propo:2:loops}
If $t>0$, then
	\begin{enumerate}
		\item \label{part1} the characteristic polynomial of $\mathcal{G}_{ttt}^{*}$ is given by
                $$P^{*}(\lambda)=(-1)^{3t} \lambda^{3(t-1)}(\lambda-2t)(\lambda^{2}-t^{2}).$$
		\item \label{part2} The characteristic polynomial of $\mathcal{G}_{tt(t+1)}^{*}$ is given by
                 $$P^{*}(\lambda)=(-1)^{3t+1}\lambda^{(3t-2)}(\lambda-t)(\lambda^{2}-t\lambda-2t(t+1)).$$
		\item \label{part3} The characteristic polynomial of $\mathcal{G}_{t(t+1)(t+1)}^{*}$ is given by
                $$P^{*}(\lambda)=(-1)^{3t+2}\lambda^{(3t-1)}(\lambda^{3}-(2t+1)\lambda^{2}-(t+1)^{2}\lambda+2t(t+1)^{2}).$$

	\end{enumerate}
\end{proposition}

\begin{proof}
We prove part (1), the proof of parts (2) and (3)  are similar (using Lemma \ref{not:integral:new}) so we omit them.
We start by observing that the adjacency matrix of $\mathcal{G}_{ttt}^{*}$, (using \eqref{eqn:3} on Page \pageref{eqn:3}) is given by

\begin{equation} \label{eqn:5}
A(\mathcal{G}_{ttt}^{*}) =  \left[ {\begin{array}{ll}
	A(K_{t}^{*}\sqcup K_{t}^{*}) & ad\,^{T} \\
	bc\,^T & A(\overline{K_{t}})
	\end{array} } \right],
\end{equation}
where $A(K_{t}^{*}\sqcup K_{t}^{*})$ is the adjacency matrix of $(K_{t}^{*}\sqcup K_{t}^{*})$, $A(\overline{K_{t}})$ is the adjacency matrix of  $\overline{K_{t}}$,
$a^{T}=[1,1,\ldots,1]_{1\times 2t}$, $b\,^{T}=[1,1,\ldots,1]_{1\times t}$, $c=[1,1,\ldots,1]_{2t\times 1}$, and $b=d$. Therefore, $ad\,^{T}$ and $bc\,^{T}$ are both
rectangular matrices of orders $2t\times t$ and $t\times 2t$, respectively where all their entries are 1. Next we observe that the characteristic polynomial of
$A(\mathcal{G}_{ttt}^{*})$ is given by \eqref{eqn:4} on Page \pageref{eqn:4}. Note that
$\widetilde{A}(K_{t}^{*}\sqcup K_{t}^{*})=ac^{T}-A(K_{t}^{*}\sqcup K_{t}^{*})$ is the adjacency matrix of the complete bipartite graph $K_{t,t}$.
From Lemma  \ref{BasicCharacteristicPoly} Part \eqref{BipartiteN} we have that the characteristic polynomial of $\widetilde{A}(K_{t}^{*}\sqcup K_{t}^{*})$
is given by $(-1)^{2t}\lambda^{2t-2}(\lambda^2-t^2)$. Similarly, we obtain that the characteristic polynomial of
$\widetilde{A}(\overline{K_{t}})=bd\,^{T}-A(\overline{K_{t}})$ is  given by $(-1)^{t}\lambda^{t-1}(\lambda-t)$.
Finally, using similar techniques as described above, we see that the characteristic polynomials of $A(K_{t}^{*}\sqcup K_{t}^{*})$ and $A(\overline{K_{t}})$
are given by $(-1)^{2t}\lambda^{2(t-1)}(\lambda-t)^{2}$ and $(-1)^{t}\lambda^{t}$, respectively.  Substituting all these characteristic polynomials in \eqref{eqn:4}
on Page \pageref{eqn:4} and simplifying, we have that the
characteristic polynomial $P^{*}(\lambda)$ of $A(\mathcal{G}_{ttt}^{*})$, is given by
\begin{multline*}
(-1)^{3t}[(\lambda)^{2t-2}(\lambda^{2}-t^{2})\lambda^{t}+(-1)^{t}\lambda^{2t-2}(\lambda-t)^{2}(-\lambda)^{t-1}(-\lambda-t)-\\
(-1)^{3t}(\lambda)^{2t-2}(\lambda^{2}-t^{2})(-\lambda)^{t-1}(-\lambda-t)].
\end{multline*}
Simplifying further we obtain  $P^{*}(\lambda)=(-1)^{3t}\lambda^{3(t-1)}(\lambda-2t)(\lambda^{2}-t^{2})$.
\end{proof}

It is clear from Proposition \ref{Propo:2:loops} Part \eqref{part1},   that the characteristic polynomial of $\mathcal{G}_{ttt}^{*}$
has only integral roots. It is also clear from Proposition \ref{Propo:2:loops} Parts \eqref{part2} and \eqref{part3}, the characteristic
polynomials for $\mathcal{G}_{tt(t+1)}^{*}$ and $\mathcal{G}_{t(t+1)(t+1)}^{*}$ do not have all integral roots.

\begin{theorem} \label{Main:thm4}
The graph $\mathcal{G}_{w}^{*}$ is integral if and only if $w=3t$ with $t\ge 1$.
\end{theorem}

\begin{proof}	
The proof that $\mathcal{G}_{w}^{*}$ is integral when $w=3t$ is straightforward from Proposition \ref{Propo:2:loops} Part \eqref{part1}.
The proof of the fact that $\mathcal{G}_{w}^{*}$ is not integral when $w=3t+1$ and $w=3t+2$, follows from  Proposition \ref{Propo:2:loops}
Parts \eqref{part2} and \eqref{part3}.
\end{proof}

\subsection{Laplacian Characteristic Polynomials}
The \emph{Laplacian} of a graph $G$ is the difference of the diagonal matrix of the vertex degrees of $G$ and the adjacency matrix of $G$.
The graph $G$ is called \emph{Laplacian integral} if all eigenvalues of its Laplacian matrix are integers.   Morris in \cite{Merris} proves  
that a cograph is Laplacian integral. Therefore, the graphs  $\mathcal{G}_{w}^{*}$ and $\mathcal{G}_{w}$,  as defined in    
\eqref{graph:definition} and  \eqref{graph:definition_noloops} on Page \pageref{graph:definition}, are Laplacian integral. The following  
proposition presents the Laplacian characteristic polynomial of these graphs.

\begin{proposition}\label{CorollaryLaplacian} For  $t>0$ these hold
	\begin{enumerate}
				
		\item   the Laplacian characteristic polynomial of $\mathcal{G}_{ttt}$ is given by
\[\mathcal{L}_{w}(\lambda)=(-1)^t\lambda(\lambda-t)(\lambda-3t)(\lambda-2t)^{3(t-1)}.
\]
\item The Laplacian characteristic polynomial of $\mathcal{G}_{tt(t+1)}$ is given by
		\[\mathcal{L}_{w}(\lambda)= (-1)^{t+1}\lambda(\lambda-(t+1))(\lambda-(3t+1))(\lambda-2t)^t(\lambda-(2t+1))^{2(t-1)}.
		\]

\item The Laplacian characteristic polynomial of $\mathcal{G}_{t(t+1)(t+1)}$ is given by
 \[\mathcal{L}_{w}(\lambda)= (-1)^{t} \lambda(\lambda-(t+1))(\lambda-(3t+2))(\lambda-2(t+1))^{t}(\lambda-(2t+1))^{2t-1}.
	\]
	\end{enumerate}
	
\end{proposition}

\section{Complement Graphs and their characteristics}\label{complement:graphs}

In this section we study the nature of the complements $\overline{\mathcal{G}}_{w}$ of the graphs $\mathcal{G}_{w}$, 
for $w=3t+r$ where  $0\le r\le 2$, given in \eqref{graph:definition_noloops} on Page \pageref{graph:definition_noloops}. 
Specifically we prove that the complement graph
$\overline{\mathcal{G}}_{w}$ of $\mathcal{G}_{w}$ is the union of two graphs, namely a complete graph and a complete bipartite graph.
We also show that $\overline{\mathcal{G}}_{ttt}$ is integral with five distinct eigenvalues while $\overline{\mathcal{G}}_{tt(t+1)}$ is
integral with four distinct eigenvalues. In the case of $\overline{\mathcal{G}}_{t(t+1)(t+1)}$, it is not integral but has five distinct eigenvalues.
We use the classic notation $K_i$ for the complete graph and $K_{i,j}$  for complete bipartite graph.
If $w=3t+r$, where $0\le r\le 2$, then

\begin{equation}\label{complement:graph}
\overline{\mathcal{G}}_{w}=
\begin{cases} K_{t,t}\sqcup K_{t},    & \mbox{ if }\;   r=0;\\
K_{t,t}\sqcup K_{t+1},  & \text{ if }\; r=1;\\
K_{t,t+1}\sqcup K_{t+1},& \text{ if }\; r=2.
\end{cases}
\end{equation}

Note that $\eqref{complement:graph}$ is in fact the complement of \eqref{graph:definition_noloops}. That is, if $r=0$ (when $r\in \{1,2\}$, it is similar),
the complement of ${\mathcal{G}}_{w}$ is given by the disjoint union of the complements of $(K_{t}\sqcup K_{t})$ and $\overline{K}_{t}$ (see \cite{barik}).
Therefore, $\overline{\mathcal{G}}_{w}=K_{t,t}\sqcup K_{t}$ since the complement of $(K_{t}\sqcup K_{t})$ is $K_{t,t}$ and the complement of
$\overline{K}_{t}$ is $K_{t}$.

\begin{theorem}
Let $w=3t+r$ with $r\in \{0,1,2\}$ and $t>0$ and let $\overline{\mathcal{G}}_{w}$ be as defined in \eqref{complement:graph}. Then these hold
\begin{enumerate}
\item if $r=0$, then $\overline{\mathcal{G}}_{w}$ is an integral graph with five distinct eigenvalues and characteristic polynomial
                $$\overline{P}(\lambda)=\lambda^{2(t-1)}(\lambda^{2}-t^{2})(\lambda-t+1)(\lambda+1)^{t-1}.$$
\item If $r=1$, then $\overline{\mathcal{G}}_{w}$ is an integral graph with four distinct eigenvalues and  characteristic polynomial
                $$\overline{P}(\lambda)=\lambda^{2(t-1)}(\lambda-t)^{2}(\lambda+t)(\lambda+1)^{t}.$$
\item If $r=2$, then $\overline{\mathcal{G}}_{w}$ is not an integral graph but it does have five distinct eigenvalues and the stated characteristic polynomial 
                $$\overline{P}(\lambda)=\lambda^{2t-1}(\lambda-t)(\lambda+1)^{t}(\lambda^{2}-t(t+1)).$$
\end{enumerate}
\end{theorem}

\begin{proof}
	We prove Part(1), the proofs of Parts (2) and (3) are similar so we omit them. Since $\overline{\mathcal{G}}_{w}=K_{t,t}\sqcup K_{t}$, for $w=3t$,
the result is straightforward using Lemma  \ref{BasicCharacteristicPoly} Part \eqref{DisjointUnionN}.
\end{proof}

\section {Hosoya graphs and their complements}\label{Hosoya:graphs}

The \emph{Hosoya triangle}, denoted by $\mathcal{H}_{F}$, is a classical triangular array where the entry in position $k$ (taken from left to right)
of the $r$th row is equal to $H_{F(r,k)}:= F_{k}F_{r-k+1}$, where $1\le k \le r$ (see Table \ref{TablaHoyaF}). The formal definition, tables, and other
results related to this triangle can be found in \cite{florezHiguitaJunesGCD, florezjunes, hosoya, koshy,koshy2}
and \cite{sloane} at \seqnum{A058071}. Note that the recursive definition of the Hosoya triangle $\mathcal{H}_{F}$ is given by \eqref{Hosoya:Seq} on 
Page \pageref{Hosoya:Seq}   with initial conditions $H_{F(1,1)}=$ $H_{F(2,1)}=$ $H_{F(2,2)}=$ $H_{F(3,2)}=1$.

\begin{table} [!ht]
\small
\begin{center} \addtolength{\tabcolsep}{-1pt} \scalebox{1}{%
\begin{tabular}{cccccccccccccccccc}
&&&&&&                                 1                                     &&&&&&&\\
&&&&&                             1   &&   1                                  &&&&&&\\
&&&&                           2 &&    1  &&   2                               &&&&&\\
&&&                         3  &&  2  &&   2  &&   3                            &&&&\\
&&                     5   &&  3  &&   4  &&   3  &&  5                           &&\\
&                   8  &&   5  &&  6  &&   6   &&   5  &&  8                       &\\
                13  &&  8  &&  10  &&  9  &&  10  &&  8  &&  13                 \\
\end{tabular}}
\end{center}
\caption{ Hosoya's triangle.} \label{TablaHoyaF}
\end{table}

In this section we characterize the graphs associated with certain symmetric matrices in the Hosoya triangle $\mathcal{H}_{F}\bmod 2$
--called Hosoya graphs-- and the  complements of these graphs. In particular, we show that the Hosoya graphs are integral 
and we present a necessary and sufficient condition for the complement Hosoya graphs to be integral.

We would like to observe that the Hosoya graphs associated with these special symmetric matrices were presented for the first time in an
article by Blair et al. \cite{Blair}.

\subsection{Graphs from symmetric matrices in the Hosoya triangle}

In this section we discuss the graphs associated with symmetric matrices in the Hosoya triangle $\mathcal{H}_{F}\bmod 2$.
Let us recall that the entries of the triangle $\mathcal{H}_{F}\bmod 2$ are given by $H_{F(r,k)}:= F_{k}F_{r-k+1}$, where $1\le k \le r$
(see Page \pageref{triangles:1}).

If $w= 3t+r$ with $t>0$ and $0\le r\le 2$, then we define the symmetric matrix $\mathcal{T}_{w}$ in the Hosoya triangle $\mathcal{H}_{F}$ in the following way:

\begin{equation} \label{eq5}
\mathcal{T}_{w} =  \left[ {\begin{array}{lllll}
	H_{F(1,1)} & H_{F(2,1)} & H_{F(3,1)} & \cdots& H_{F(w,1)} \\
	H_{F(2,2)}  &  H_{F(3,2)}& H_{F(4,2)} & \cdots & H_{F(w+1,2)}\\
	\vdots &\vdots&\vdots&\ddots &\vdots\\
	H_{F(w,w)}  &  H_{F(w+1,w)} & H_{F(w+2,w)} & \cdots & H_{F(2w-1,w)}
	\end{array} } \right]_{w\times w}.
\end{equation}

Let $T_{w}$ be the matrix $\mathcal{T}_{w} \bmod 2$. That is, $T_{w}$ is the matrix $[t_{ij}]_{w\times w}$ where $t_{ij}=H_{F(i,j)} \bmod 2$.
Let $\Theta_{w}$ be the graph of ${T}_{w}$, see Table \ref{tabla3}.
In Proposition \ref{Hosoya:graph} we show that $\Theta_{w}$ is integral.
Note that if $w=3t+r$ with $t> 0$ and $0\le r \le 2 $, then $(T_{w}-\lambda I)$ is a product of two vectors. It is easy to see
that $\Theta_{w}$ is an integral graph with exactly one non-zero eigenvalue $\lambda=2t+r$.

\begin{proposition}\label{Hosoya:graph}
 Let  $\Theta_{w}$ be the graph from $T_{w}$.
If $w=3t+r$ with $t> 0$ and $0\le r \le 2 $, then $\Theta_{w}$ is the complete graph $K_{2t+r}$ with loops at every vertex and $t$ isolated vertices.
Furthermore, $\Theta_{w}$ is an integral graph.
\end{proposition}

\begin{proof} It is known that the Fibonacci number $F_{n}\equiv 0 \bmod 2 \iff 3\mid n$. This and the definition of $\mathcal{T}_{w}$ (see \eqref{eq5}),
imply that every
	third row and every third column of $\mathcal{T}_{w}$ are formed by even numbers and that the remaining rows and columns are formed by odd numbers only.
	Thus, if $t_{ij}$ is an entry of $\mathcal{T}_{w}$, then $t_{ij}\equiv 0 \bmod 2 \iff i \equiv 0 \mod 3 \text{ or } j \equiv 0 \mod 3$.
	This and $w=3t+r$ imply that $T_{w}$ contains $t$ columns and $t$ rows with zeros as entries. The remaining $2t+r$ rows and columns have ones as entries.
	These two features of $T_{w}$  give   $(K_{2t+r}^{*}\sqcup \overline{K_{t}})$ (the disjoint union of the complete graph on $2t+r$ vertices with loops at every
	vertex and $t$ isolated vertices). Finally, from the discussion above we have that the only eigenvalue of $\Theta_{w}$ is an integer, namely $\lambda=2t+r$.
Therefore, $\Theta_{w}$ is an integral graph.
	This completes the proof.
\end{proof}

\begin{table}[!ht]
	\begin{tabular}{|l|c|c|c|} \hline
	$k$ &	$3t$ & $3t+1$   &  $3t+2$ \\ \hline  \hline
	1  & $\vcenter{\hbox{\includegraphics[scale=.3]{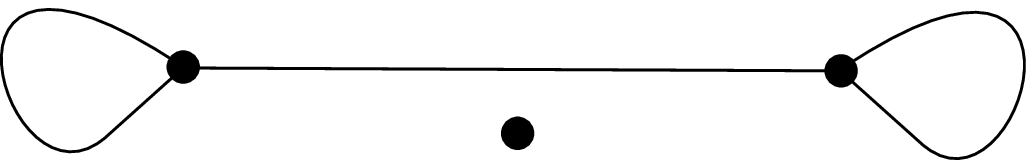}}}$ & 
	$\vcenter{\hbox{\includegraphics[scale=.085]{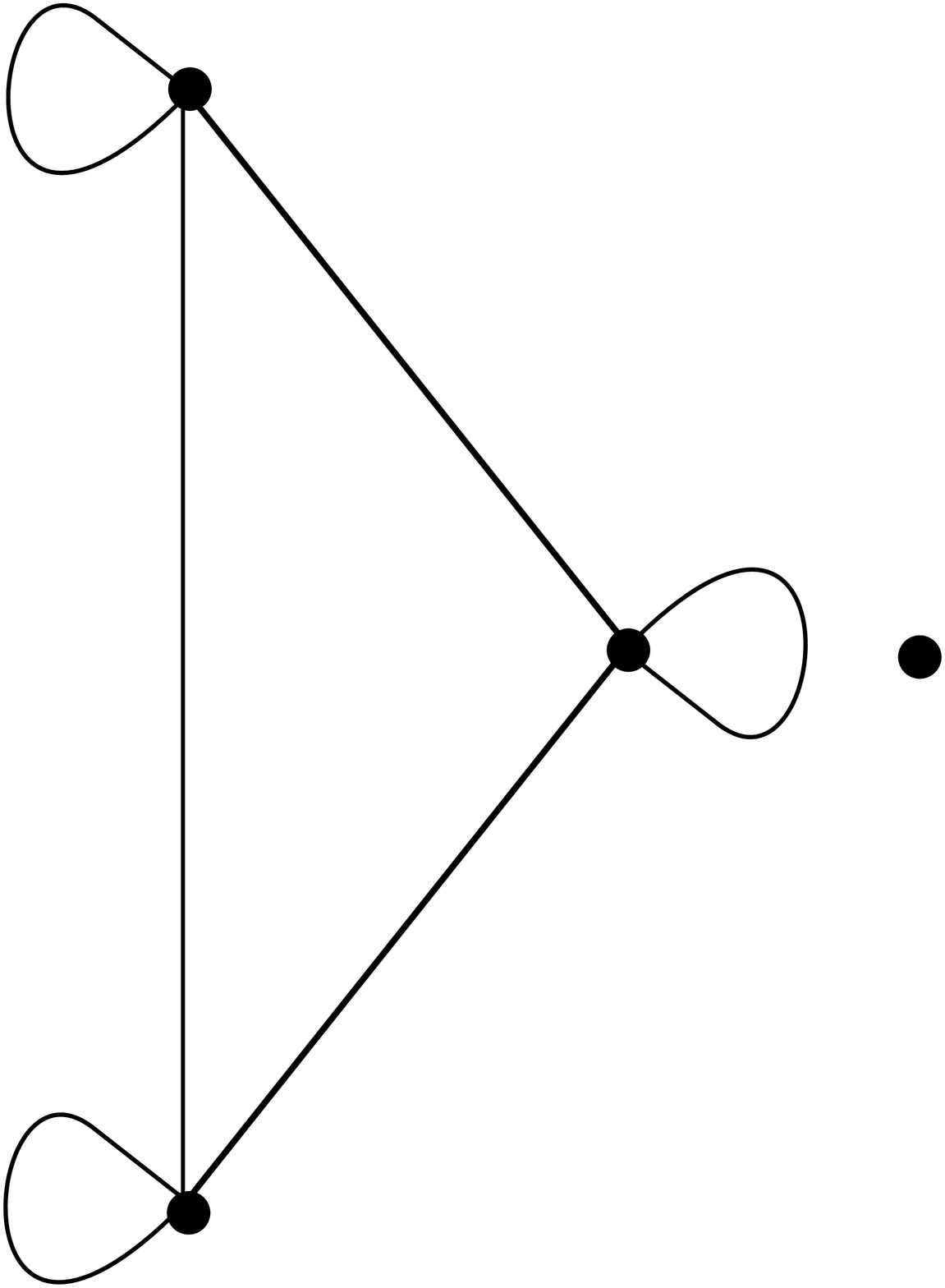}}}$&  
	$\vcenter{\hbox{\includegraphics[scale=.3]{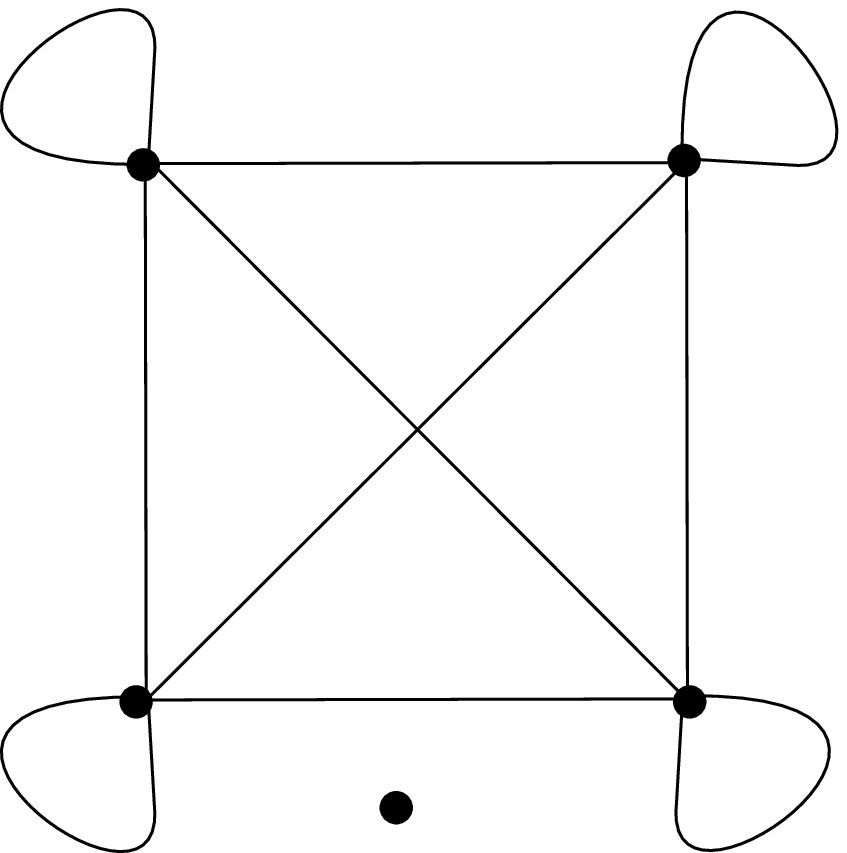}}}$\\ \hline
    2 & $\vcenter{\hbox{\includegraphics[scale=.3]{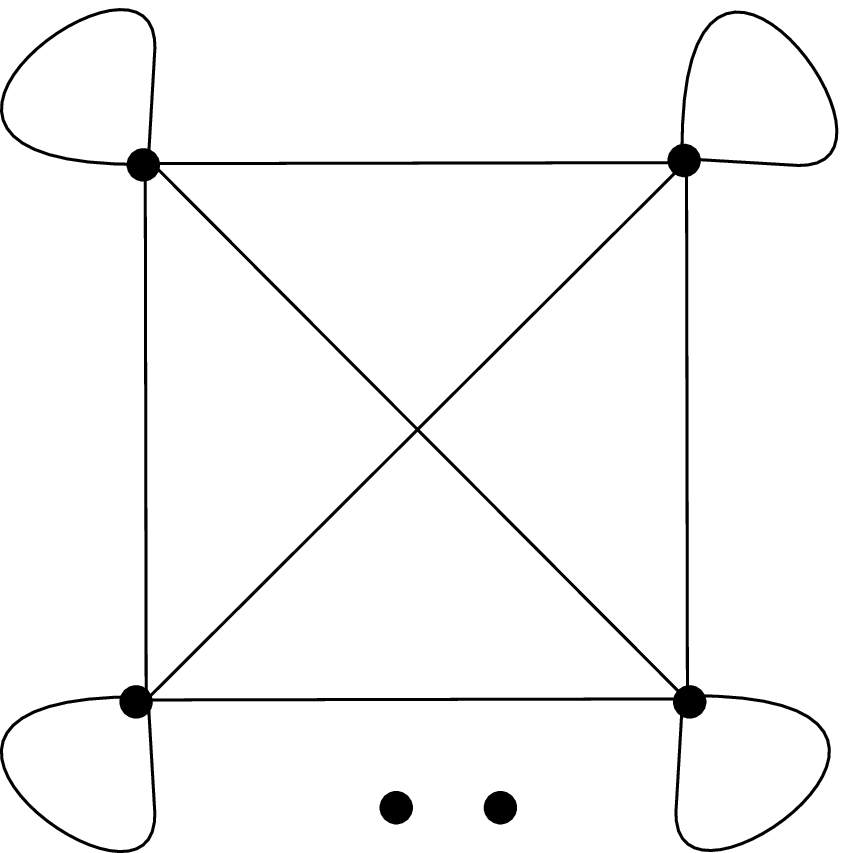}}}$ &
    $\vcenter{\hbox{\includegraphics[scale=.11]{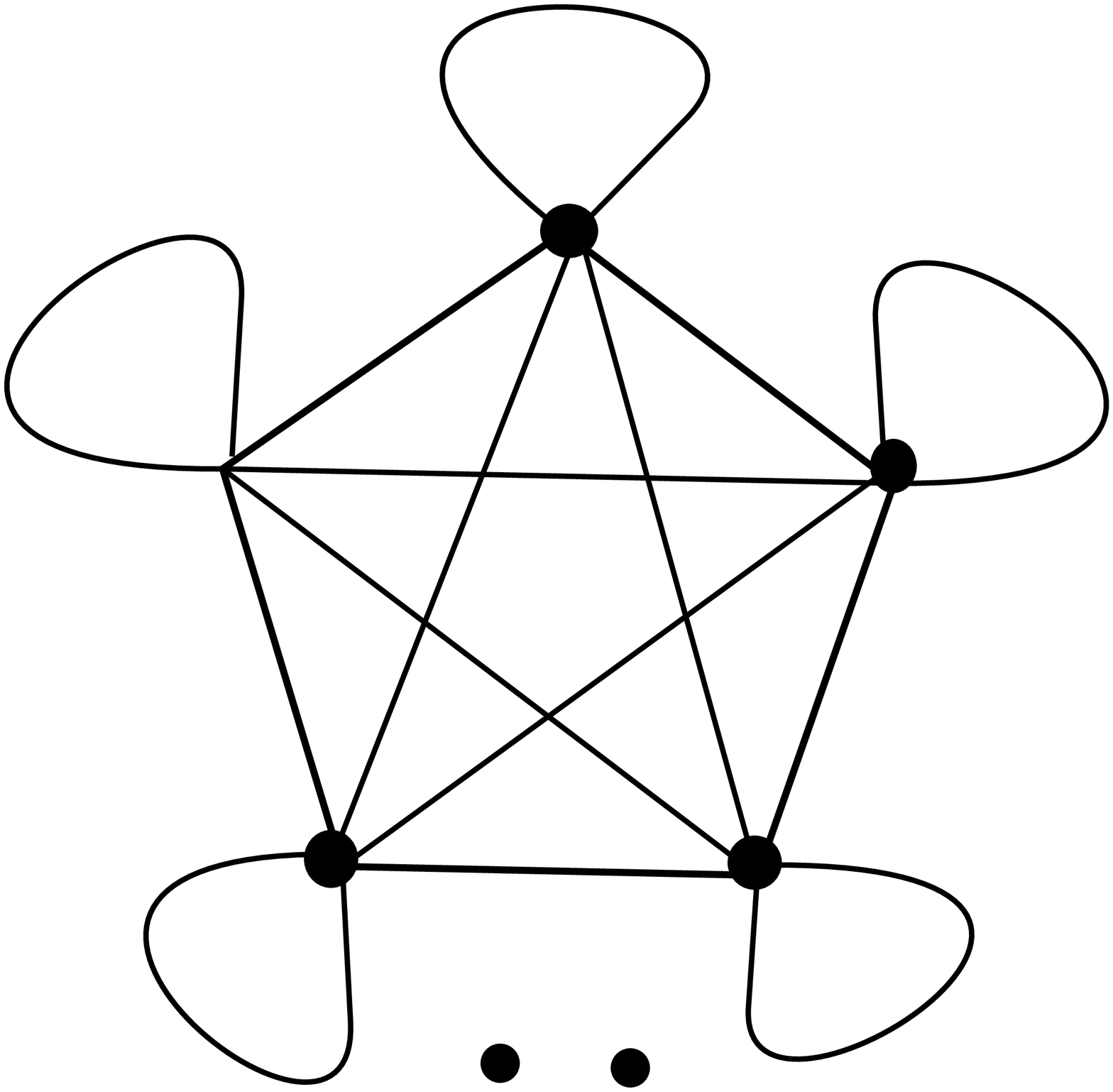}}}$& 
    $\vcenter{\hbox{\includegraphics[scale=.3]{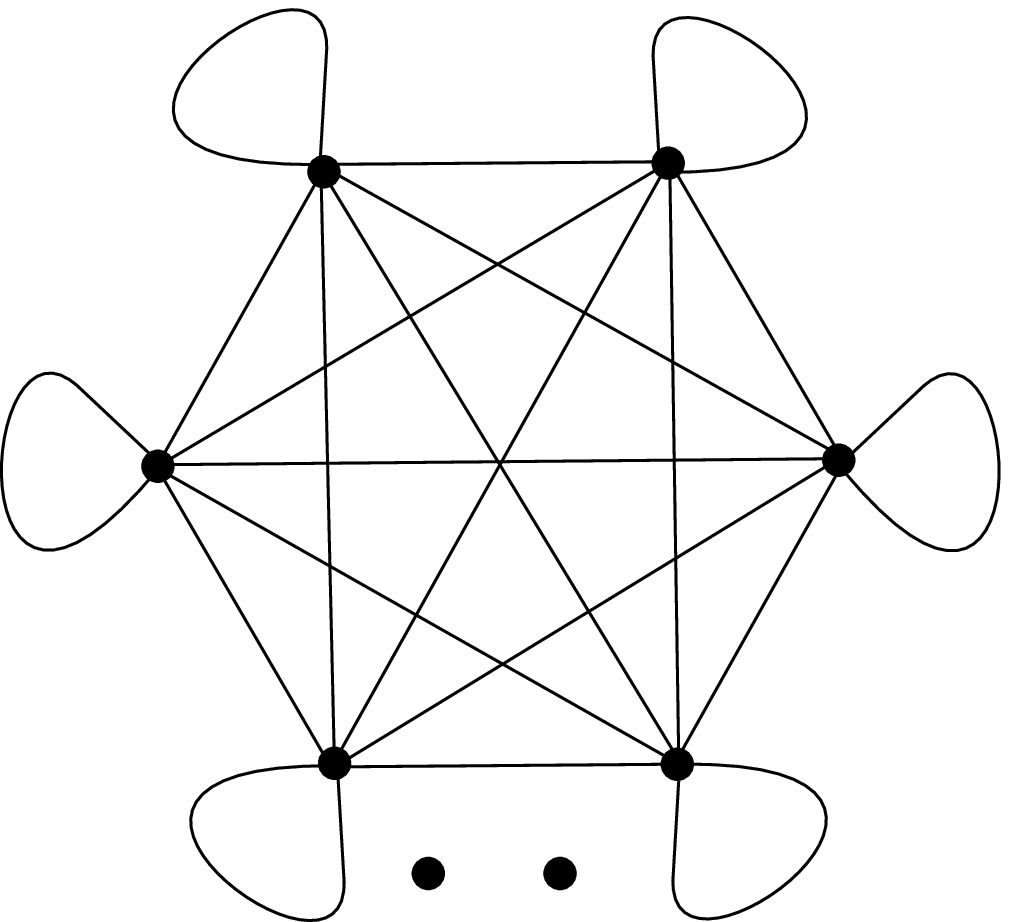}}}$\\ \hline
    3 & $\vcenter{\hbox{\includegraphics[scale=.33]{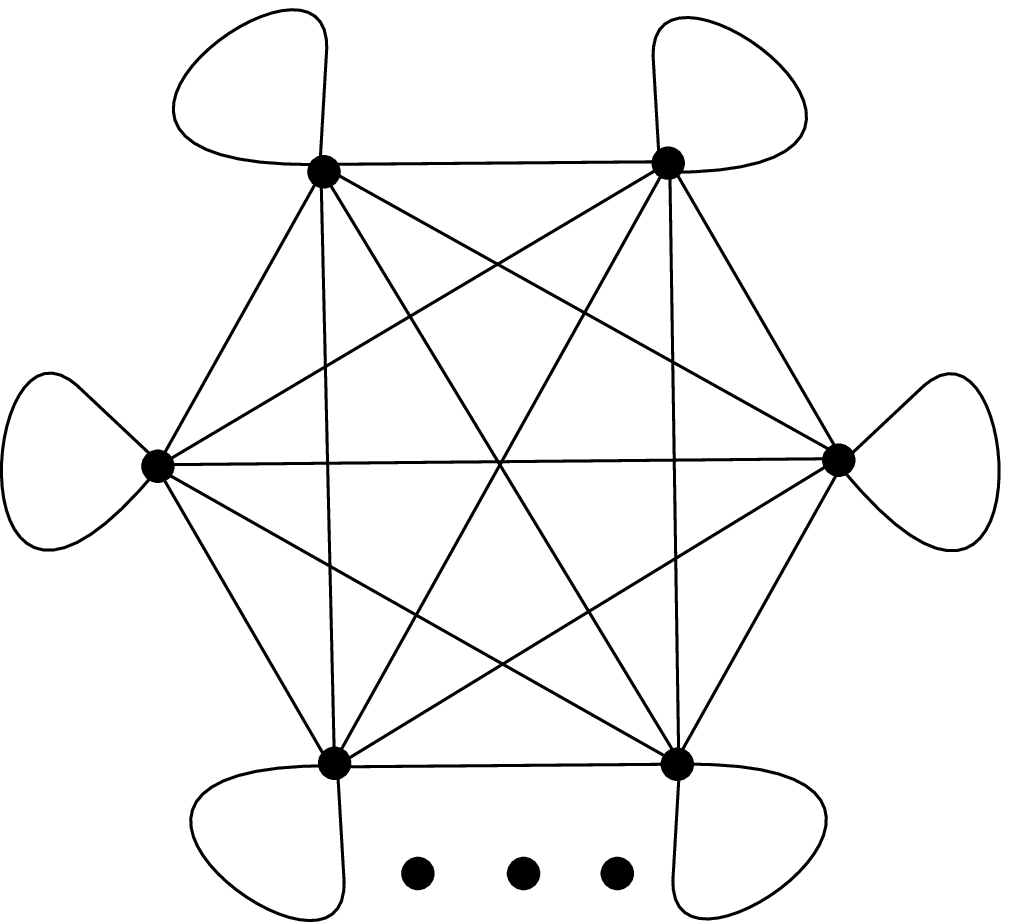}}}$ &
     $\vcenter{\hbox{\includegraphics[scale=.3]{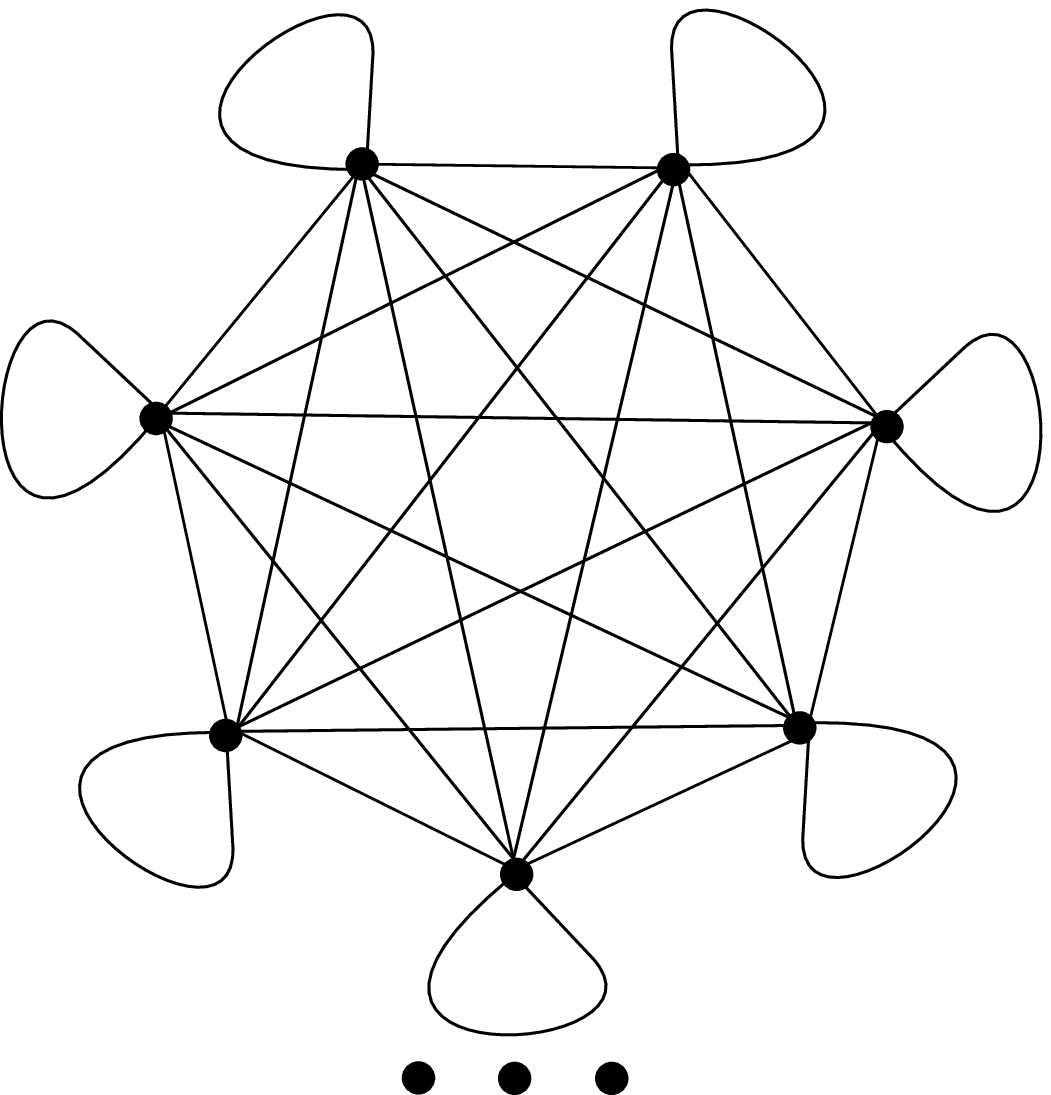}}}$
     &$\vcenter{\hbox{\includegraphics[scale=.3]{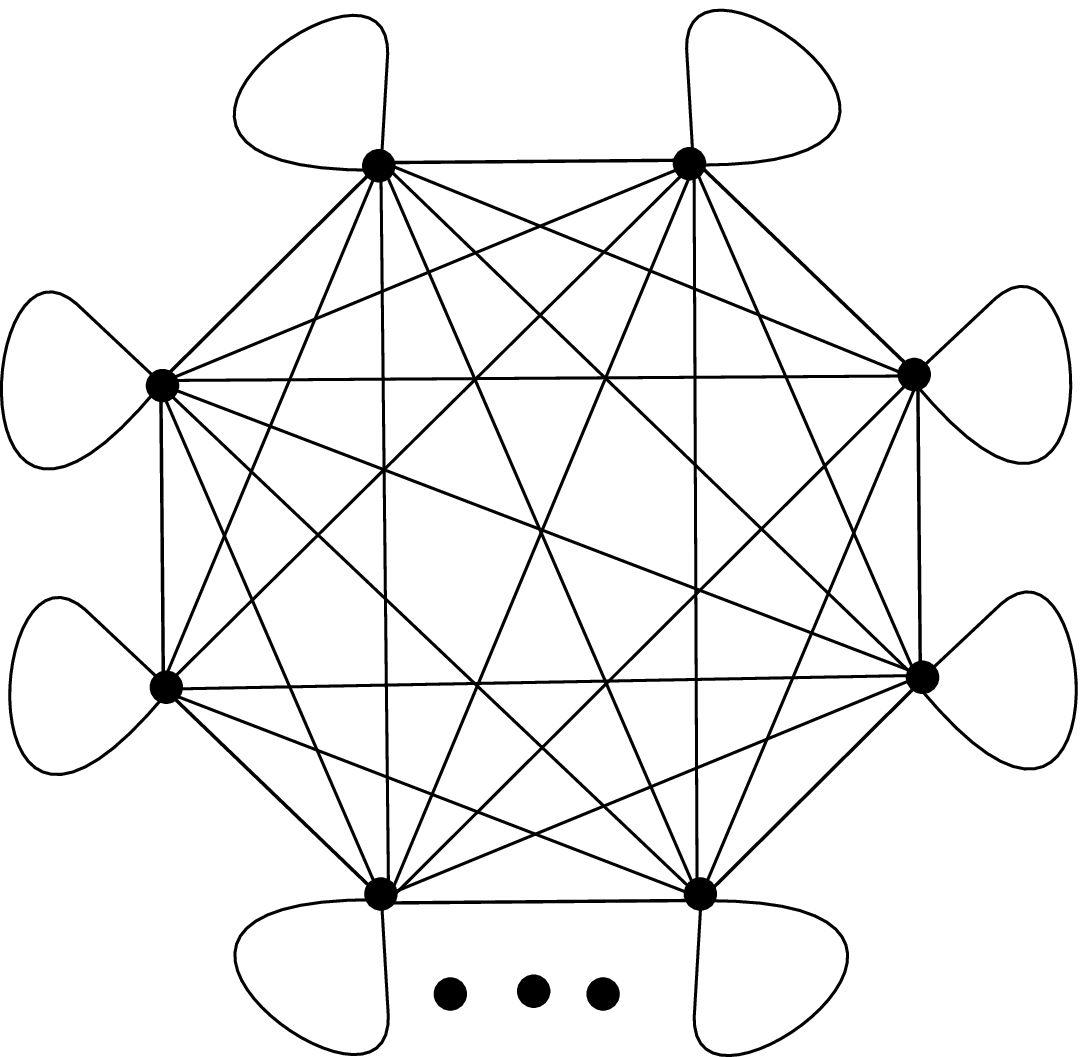}}}$\\ \hline
    4 & $\vcenter{\hbox{\includegraphics[scale=.3]{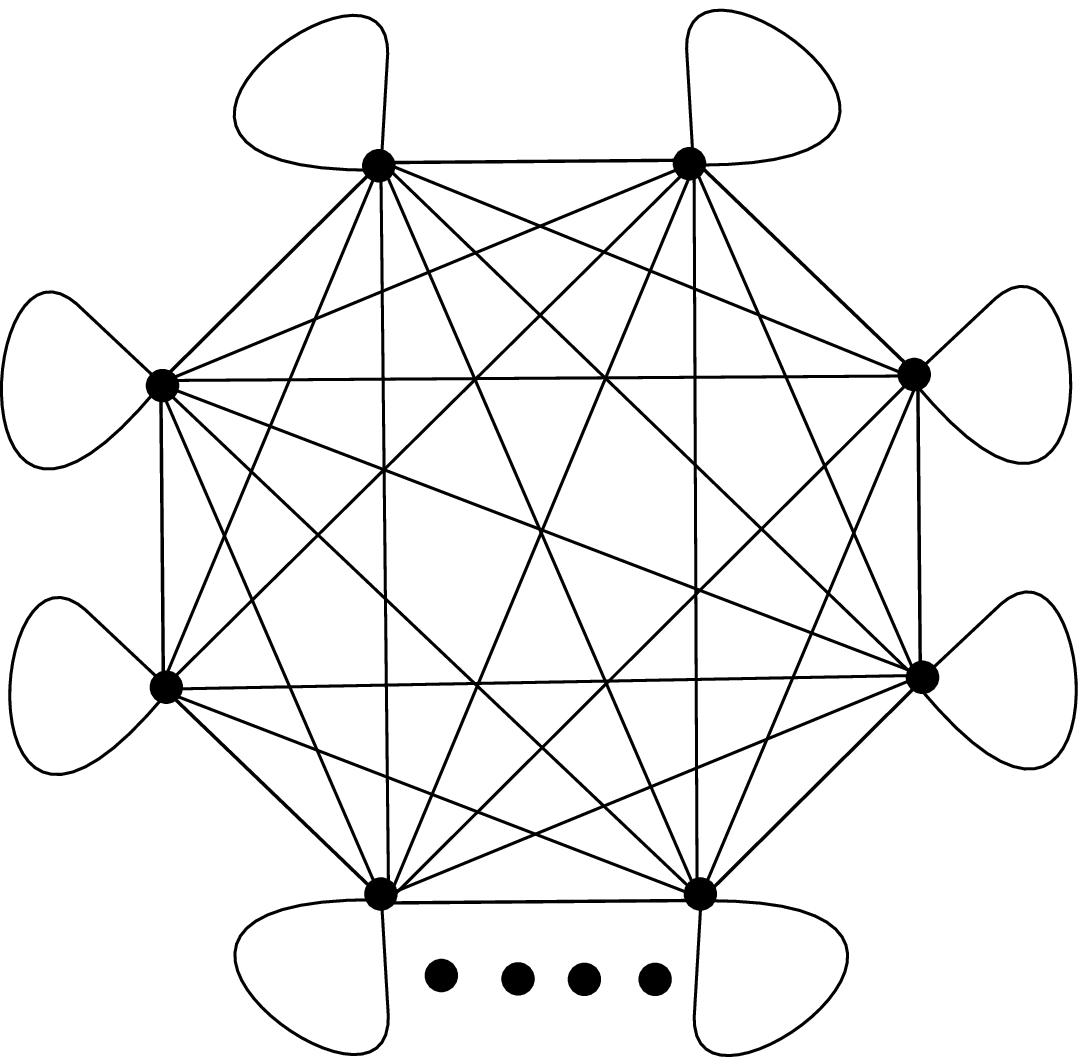}}}$ &
     $\vcenter{\hbox{\includegraphics[scale=.34]{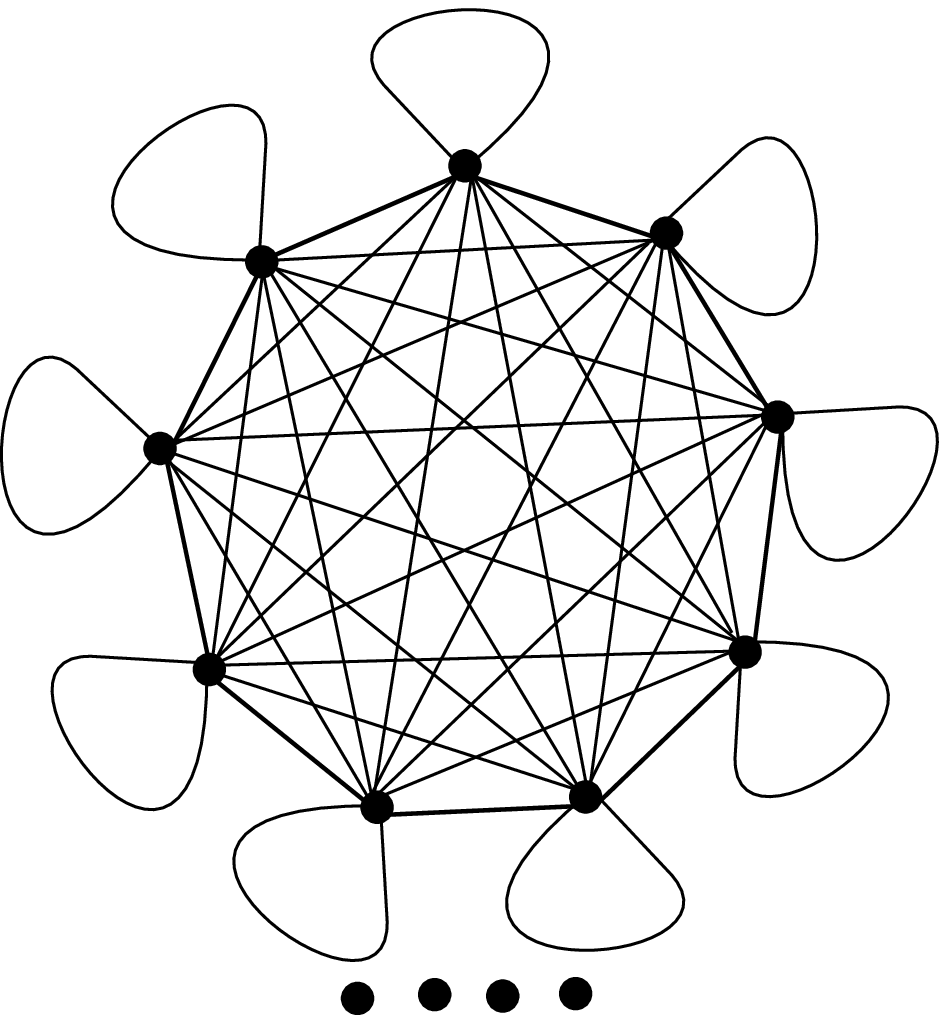}}}$ & 
     $\vcenter{\hbox{\includegraphics[scale=.33]{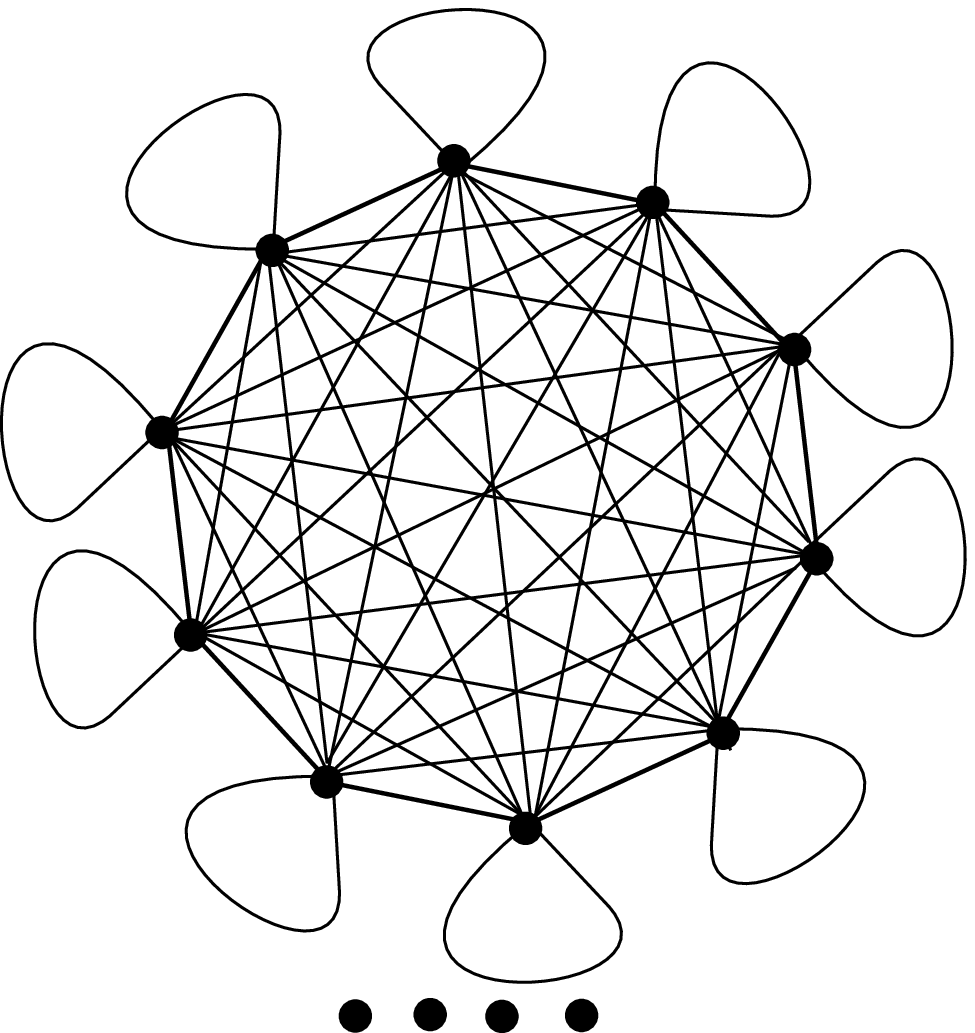}}}$\\ \hline
\end{tabular}
\caption{Graphs $\Theta_{w}$.} \label{tabla3}
\end{table}

\subsection{Complements of the Hosoya Graph}

In this section we first discuss the necessary and sufficient condition for the graphs of the form
$G:= K_{n}\nabla \overline{K}_{r}$ for $n,r\in \mathbb{Z}_{>0}$, to be integral. These graphs are the generalizations of graphs that are the
complements of the Hosoya graphs. It is easy to see (using \cite[Table 2]{barik}) that the graph $G:=K_{n} \nabla \overline{K}_{r}$ is integral if and only if there are $p, q\in \mathbb{Z}_{>0}$ such that $nr=pq$ with $n-1=p-q$. The characteristic polynomial of $G$ is given by  $P(\lambda)=(-1)^r\lambda^{r-1} (\lambda+1)^{(n-1)}(\lambda^2-(n-1)\lambda-nr)$.

We observe that the complement of the Hosoya graphs, denoted by $\overline{\Theta}_{w}$, can be described as the join of a complete graph and an empty graph.
In fact, $\overline{\Theta}_{w}= K_{t}\nabla\overline{K}_{w-t}$ with $w=3t+r$ for $0\le r\le 2$.
These graphs are special cases of the graphs $K_{n}\nabla \overline{K}_{r}$ (discussed in the previous paragraph).
In addition, recall that the characteristic polynomial of
$K_{n}\nabla \overline{K}_{r}$ is $P(\lambda)=(-1)^r\lambda^{r-1} (\lambda+1)^{(n-1)}(\lambda^2-(n-1)\lambda-nr)$.
Thus, if we consider the graphs $\overline{\Theta}_{w}$, then it is easy to check that $P(\lambda)$ has integral roots if and only if $w=3t+2$.
In other words, the graph $\overline{\Theta}_{w}$ is integral if and only if $w=3t+2$.  Recently, Be\'{a}ta et al. \cite{beata}
have shown a combinatorial connection with graphs of the  form $K_{n} \nabla \overline{K}_{r}$.

\section{Acknowledgement}
	
The second and third authors were partially supported by The Citadel Foundation. 

\bigskip
\hrule
\bigskip
	
\noindent 2010 {\it Mathematics Subject Classification}:
Primary 68R11; Secondary 11B39.
	
\noindent \emph{Keywords: }
Fibonacci number, determinant Hosoya triangle, adjacency matrix, eigenvalue, integral graph, cograph.

\end{document}